\documentclass{article}
\usepackage{amssymb, amsmath, amsthm}
\usepackage{graphicx,verbatim}
\usepackage{rotating}
\usepackage[all,cmtip]{xy}
\usepackage{url}
\usepackage{mathrsfs}
\usepackage{bbold} 
\usepackage[colorlinks = true]{hyperref}
\usepackage[normalem]{ulem}
\usepackage{enumitem}
\usepackage[colorinlistoftodos]{todonotes}
\usepackage{appendix}

\usepackage[a4paper,left=1in,width=458pt,textheight=650pt,top=1.5in]{geometry}

\usepackage{tikz}
\usetikzlibrary{calc}
\usepgflibrary{shapes.geometric}
\usepgflibrary{shapes.misc}
\usetikzlibrary{positioning}
\usetikzlibrary{decorations}
\usetikzlibrary{arrows}

    \theoremstyle{plain}
    \newtheorem{theorem}{Theorem}
    \newtheorem{proposition}[theorem]{Proposition}
    \newtheorem{corollary}[theorem]{Corollary}
    \newtheorem{lemma}[theorem]{Lemma}
    \newtheorem{conjecture}{Conjecture}
    \newtheorem*{theorem*}{Theorem}
    \newtheorem*{proposition*}{Proposition}
    \newtheorem*{corollary*}{Corollary}
    \newtheorem*{lemma*}{Lemma}
    \newtheorem*{conjecture*}{Conjecture}

    \theoremstyle{definition}
    \newtheorem{definition}[theorem]{Definition}
    \newtheorem{example}[theorem]{Example}
    \newtheorem*{definition*}{Definition}
    \newtheorem*{example*}{Example}

    \theoremstyle{remark}
    \newtheorem{remark}[theorem]{Remark}
    \newtheorem*{remark*}{Remark}

    \renewcommand{\epsilon}{\varepsilon}

    \newcommand{\bN}{\mathbb{N}}
    \newcommand{\nat}{\mathbb{N}}

    \newcommand{\bZ}{\mathbb{Z}}
    \newcommand{\integ}{\mathbb{Z}}

    \newcommand{\bG}{\mathbb{G}}

    \newcommand{\cO}{\mathcal{O}}

    \newcommand{\cF}{\mathcal{F}}
    \newcommand{\cG}{\mathcal{G}}
    \newcommand{\cL}{\mathcal{L}}
    \newcommand{\cA}{\mathcal{A}}
    \newcommand{\cB}{\mathcal{B}}

    \newcommand{\cP}{\mathcal{P}}
    \newcommand{\cT}{\mathcal{T}}

    \newcommand{\sC}{\mathscr{C}}

    \newcommand{\sT}{\mathscr{T}}
    \newcommand{\GT}{\mathscr{GT}}
    \newcommand{\gtpn}{\GT(\pro{n})}

    \newcommand{\im}{\operatorname{im }}
    \renewcommand{\ker}{\operatorname{ker }}

    \newcommand{\Res}{\operatorname{Res}}
    
    \newcommand{\Gr}{\operatorname{Gr}}
    \newcommand{\GR}{\operatorname{GR}}
    \newcommand{\GrEnd}{\operatorname{GrEnd}}
    
    \newcommand{\Hom}{\operatorname{Hom}}

    \newcommand{\Aut}{\operatorname{Aut}}

    \newcommand{\QGr}{\operatorname{QGr}}

    \newcommand{\Tors}{\operatorname{Tors}}

    \newcommand{\pro}[1]{\mathbb{P}^{#1}}

    \newcommand{\op}[1]{{#1}^{\textup{op}}}
    \newcommand{\iso}{\cong}
    
    \newcommand{\id}{\operatorname{id}}

     \newcommand{\ncpn}{\mathscr{NC}\paren{\pro{n}}}
     \newcommand{\ncp}[1]{\mathscr{NC}\paren{\pro{#1}}}

     \newcommand{\sqbinom}[2]{\genfrac{[}{]}{0pt}{1}{#1}{#2}}

     \newcommand{\defeq}{\stackrel{\scriptscriptstyle{\mathrm{def}}}{=}}

    \renewcommand{\iff}{\Leftrightarrow}
    
    \newcommand{\brac}[1]{\ensuremath{\left\{#1\right\}}}
    \newcommand{\paren}[1]{\ensuremath{\left(#1\right)}}
    \newcommand{\sqbrac}[1]{\ensuremath{\left[#1\right]}}

    \newcommand{\git}{\mathbin{
  \mathchoice{/\mkern-6mu/}
    {/\mkern-6mu/}
    {/\mkern-5mu/}
    {/\mkern-5mu/}}}
    \renewcommand\bar\overline
    \renewcommand\tilde\widetilde

    \newcommand{\PGL}[2]{\operatorname{PGL}_{#1}\!\paren{#2}}

\title{Automorphism groupoids in noncommutative projective geometry}
\author{Nicholas Cooney
\\ \small{\textit{Max-Planck-Institut f\"ur Mathematik,}}
\\ \small{\textit{Vivatsgasse 7, 53111 Bonn, Germany}}
\and  Jan E. Grabowski\footnotemark[3] 
\\ \small{\textit{Department of Mathematics and Statistics, Lancaster University,}}
\\ \small{\textit{Lancaster, LA1 4YF, United Kingdom}}
}
\date{21st April 2022}

\begin{document}
\maketitle

\renewcommand{\thefootnote}{\fnsymbol{footnote}}
\footnotetext[3]{Email: \url{j.grabowski@lancaster.ac.uk}.  Website: \url{http://www.maths.lancs.ac.uk/~grabowsj/}}
\renewcommand{\thefootnote}{\arabic{footnote}}
\setcounter{footnote}{0}

\begin{abstract} 
  We address a natural question in noncommutative geometry, namely the rigidity observed in many examples, whereby noncommutative spaces (or equivalently their coordinate algebras) have very few automorphisms by comparison with their commutative counterparts.

  In the framework of noncommutative projective geometry, we define a groupoid whose objects are noncommutative projective spaces of a given dimension and whose morphisms correspond to isomorphisms of these.  This groupoid is then a natural generalization of an automorphism group.  Using work of Zhang, we may translate this structure to the algebraic side, wherein we consider homogeneous coordinate algebras of noncommutative projective spaces.  The morphisms in our groupoid precisely correspond to the existence of a Zhang twist relating the two coordinate algebras.

  We analyse this automorphism groupoid, using the geometry of the point scheme, as introduced by Artin-Tate-Van den Bergh, to relate morphisms in our groupoid to certain automorphisms of the point scheme.

  We apply our results to two important examples, the two-dimensional quantum projective space and Sklyanin algebras.  In both cases, we are able to use the geometry of the point schemes to fully describe the corresponding component of the automorphism groupoid.  This provides a concrete description of the collection of Zhang twists of these algebras.
  
  \vspace{1em}

\end{abstract}

\vfill \vfill
\pagebreak

\section{Introduction}

It is well-known that affine algebraic varieties can have many automorphisms, but that noncommutative deformations of these typically have very few automorphisms.  For example, let $q\in k^{\ast}$, $q$ not a root of unity, and denote by $\mathcal{O}_{q}(\mathbb{P}_{k}^{n})$ the $k$-algebra
\[ \mathcal{O}_{q}(\mathbb{P}_{k}^{n})=k\langle x_{0},\dotsc ,x_{n} \rangle/\langle x_{i}x_{j}=qx_{j}x_{i}\ \forall\ i<j \rangle. \]
Then Alev and Chamarie (\cite{AlevChamarie}) have shown that
\[ \Aut(\mathcal{O}_{q}(\mathbb{P}_{k}^{n}))=\begin{cases} (k^{*})^{n+1} & n\neq 2, \\ k\ltimes (k^{\ast})^{3} & n=2. \end{cases} \]

The same phenomenon of a significantly smaller automorphism group has been observed for many other quantum algebras: quantized enveloping algebras and related algebras, quantum matrices, quantum Weyl algebras, Nichols algebras, and others.  Contributors include Andruskiewitsch-Dumas \cite{AndruskiewitschDumas}, Fleury \cite{Fleury}, Goodearl-Yakimov \cite{GoodearlYakimovAuts}, Joseph \cite{Joseph}, Launois-Lenagan \cite{LaunoisLenaganAutsQM}, Rigal \cite{Rigal} and Yakimov \cite{YakimovLLConj},\cite{YakimovADConj}.

This suggests that the notion of automorphism might be too strong, and that looking at other ways in which noncommutative spaces can be regarded as the same might ``recover'' some of the lost symmetries.  To do this we will follow the philosophy of noncommutative geometry and focus on (graded) module categories as noncommutative spaces.

As we will see, noncommutative algebras sharing homological properties with a given commutative algebra have the property that module categories associated to them will be equivalent and even isomorphic - induced by so-called twists of these algebras.  We wish to remember \emph{how} these categories are equivalent and so we encode this data in a groupoid, rather than a group.

We will take a collection of noncommutative spaces that model our chosen classical space and, instead of allowing all possible morphisms between them, concentrate on just the isomorphisms between them.  Such a groupoid typically has many connected components.  This will be the basis for the construction we describe, which we approach through noncommutative projective geometry.

In this setting, certain graded module categories are considered to be noncommutative projective schemes. We define the class of algebras that we are interested in, and these will play the role of homogeneous coordinate rings of our noncommutative projective spaces.  We then recall the definition of a Zhang twist of a graded algebra. This is a family of automorphisms of the underlying vector space which induce a new associative multiplication, ``twisting" the original multiplication. We will need to restrict our attention to geometric algebras, as defined by Mori (\cite{Mori}): for these algebras, there is a natural subscheme $E$ of $\pro{n}$ and automorphism $\sigma$ of $E$ whose role we explain shortly.

We introduce and define our main object of study, the groupoid $\ncpn$, in Section~\ref{s:ncpn}. This is a groupoid whose objects are certain module categories for the coordinate rings of geometric noncommutative $\mathbb{P}^{n}$s and whose morphisms are particular equivalences of these categories induced by Zhang twists. That this is a groupoid is deduced from a result of Zhang, that twisting preserves the properties of a noncommutative $\mathbb{P}^{n}$. We study slice categories of $\ncpn$, which allow us to study all the twists - that is, all of the ``generalized automorphisms" - of a given noncommutative $\mathbb{P}^{n}$.

We then link Zhang twists of geometric algebras with automorphisms of their associated point schemes, for arbitrary $n$. Let $\mathcal{A}$ be an object in $\ncpn$ associated to a noncommutative $\mathbb{P}^{n}$ which is geometric in Mori's sense. Let $E\subset\mathbb{P}^{n}$ be its associated point scheme, $\sigma$ a certain automorphism of $E$ and $\mathcal{L}$ the very ample line bundle giving the embedding into $\mathbb{P}^{n}$; these are packaged together to give the geometric triple $(E,\sigma,\cL)$ associated to $\cA$.  Twists of geometric noncommutative $\pro{n}$s correspond to geometric systems, certain families of maps between triples.

Let $\Aut(E \uparrow \pro{n})$ be the subgroup of automorphisms of the point scheme that extend to $\pro{n}$ and let $\mathrm{Res}_{E}\colon \Aut(\pro{n} \downarrow E)\to \Aut(E \uparrow \pro{n})$ denote the homomorphism taking an automorphism of $\pro{n}$ that preserves $E$ to its restriction to $E$.

The main result of the paper is the following, Theorem~\ref{t:aut-ncpn}:

\begin{theorem*} Let $\cA$ be a geometric noncommutative $\pro{n}$ with associated geometric triple $(E,\sigma,\cL)$.  Then the twists of $\cA$ are parameterised by a subset of $\Aut(E\uparrow \pro{n})\sigma \times (\ker \mathrm{Res}_{E})^{\nat}$ (up to isomorphism of the associated homogeneous coordinate ring).
\end{theorem*}

In important examples, $\mathrm{Res}_{E}$ is injective: in such cases, we see that twists are parameterised (up to isomorphism) by a subset of the coset $\Aut(E \uparrow \pro{n})\sigma$. We can also give a geometric condition for the subset corresponding to twists being all of $\Aut(E\uparrow \pro{n}) \times (\ker \mathrm{Res}_{E})^{\nat}$.

The theorem therefore says that the (generalized) automorphisms of a geometric noncommutative $\pro{n}$ are controlled by the geometry of the point scheme with its embedding in $\pro{n}$.  We conclude with three examples of the theorem: the commutative case of $\cO(\pro{n})$; quantum deformations of polynomial rings in $n$ variables; and Sklyanin algebras, an important class of noncommutative $\pro{2}$s.  

\subsection*{Acknowledgements}

The authors would like to thank Yemon Choi, Andrew Davies, Jonny Evans, Sian Fryer, Sue Sierra and Toby Stafford for helpful discussions during the development of this work, as well as the many seminar audience members whose questions aided our efforts to clarify the presentation of the results herein.

Financial support from the EPSRC (EP/M001148/1) and the ERC project ModRed is gratefully acknowledged. The first author is grateful to the Max Planck Institute for Mathematics for its hospitality and financial support during the preparation of this paper.

\section{Preliminaries on noncommutative projective geometry}\label{s:preliminaries}

\subsection{Noncommutative projective spaces}

Let $k$ be an algebraically closed field of characteristic $0$ and let $A$ be a right Noetherian $\integ$-graded $k$-algebra. Let $\Gr A$ denote the category of (not necessarily finitely generated) $\integ$-graded right $A$-modules with morphisms of degree $0$. Henceforth, by any graded module category over a $k$-algebra $A$, we mean the category of right $A$-modules. We denote by $(1)$ the grading shift functor on $\Gr A$; this functor is an auto-equivalence. That is, for $M$ a $\integ$-graded right $A$-module, and $n\in\integ$ denote by $M(n)$ the graded module whose degree $m$ component is given by $M(n)_{m}=M_{n+m}$. Recall from~\cite{ArtinZhang} that, for $M$ in $\Gr A$, an element $x\in M$ is called \textit{torsion} if $xA_{\geq s}=0$ for some s. The set of torsion elements of $M$ forms a graded $A$-submodule, $\tau(M)$. The module $M$ is called torsion if $\tau(M)=M$ and torsion-free if $\tau(M)=0$. The submodule $\tau(M)$ is the smallest such that the quotient $M\slash\tau(M)$ is torsion-free. The collection of all torsion modules forms a Serre subcategory of $\Gr A$ which we denote $\Tors (A)$.

Let $\pi:\Gr A\to \Gr A\slash\Tors (A)$ be the canonical quotient functor and define $\QGr A \defeq\Gr A\slash\Tors (A)$. One can roughly think of $\QGr A$ as the category whose objects are the same as those of $\Gr A$, but with torsion modules isomorphic to the zero module. More precisely, the objects of $\QGr A$ are those of $\Gr A$ and the morphisms can be described as
\[
\Hom_{\QGr A}(\pi(M),\pi(N))\cong\varinjlim\Hom_{\Gr A}(M^{\prime},N\slash\tau(N))
\]
where the limit runs over the quasi-directed category of submodules $M^{\prime}$ of $M$ such that $M\slash M^{\prime}$ is torsion. Note that the shift functor $(1)$ descends to $\QGr A$; we will abuse notation and denote the shift on $\QGr A$ by $(1)$ also. See~\cite{ArtinZhang} and \cite{StaffVdB} for more details. 

We now restrict to considering noncommutative analogues of $\pro{n}$. That is, we want to define the class of noncommutative $k$-algebras that we consider as giving rise to \textit{noncommutative projective $n$-space} - triples of the form $\paren{\QGr A, \pi A, (1)}$ where the $k$-algebra $A$ shares important homological and algebraic properties with a graded polynomial ring in $n+1$ variables. We adopt a variant of the definition given in section 4 of \cite{Keeler}.  We direct the reader to the survey \cite{StaffVdB} for definitions not included here as well as more detailed intuition and motivation behind this definition.

\begin{definition}[{cf.\ \cite[Definition 4.3]{Keeler}}]\label{Keeler} Let $A$ be a connected graded $k$-algebra. We say that the triple $\paren{\QGr A, \pi A, (1)}$ is a \emph{noncommutative} $\pro{n}$ and that $A$ is \emph{the (homogeneous) coordinate ring of a noncommutative} $\pro{n}$ if
\begin{enumerate}[label=\textup{(\roman*)}]
\item $A$ is a Noetherian domain;
\item $A$ is AS-regular of dimension $n+1$;
\item $A$ is generated in degree 1; and
\item $A$ has Hilbert series $\paren{1-t}^{-n-1}$.
\end{enumerate}
\end{definition}

As Keeler notes, his characterization is more restrictive than is found elsewhere in the literature.

\begin{remark}
Note that we have reordered the conditions in Keeler's definition so that $A$ being Noetherian is placed first. In this case, $A$ has equal left and right global dimensions, and equal left and right injective dimensions. Moreover, the definition of AS-regularity is left-right symmetric and so we can omit the conditions placed on $A^{op}$ by Keeler.
\end{remark} 

\subsection{Zhang twists}

We recall the definitions related to the notion of a Zhang twist, the relevance of which will become clear shortly.

\begin{definition}[{\cite[Definition 2.1]{Zhang}}] Let $\bG$ be a semigroup with identity $e$ and let $A$ be a $\bG$-graded $k$-algebra. A \emph{twisting system} for $A$ is a set of graded $k$-linear $A$-automorphisms $\tau=\brac{\tau_{g} \mid g\in\bG}$ such that $$\tau_{g}\paren{y\tau_{h}\paren{z}}=\tau_{g}\paren{y}\tau_{gh}\paren{z}$$ for all $g,h,l\in\bG$, $y\in A_{h}$ and $z\in A_{l}$.
\end{definition}

Any semigroup homomorphism $\bG\rightarrow\Aut_{\bG}\paren{A}$ given by $g\mapsto\tau_{g}$ produces a twisting system $\brac{\tau_{g} \mid g\in\bG}$ where $\Aut_{\bG}\paren{A}$ denotes the group of graded algebra automorphisms of $A$. Twisting systems arising from semigroup homomorphisms as above are called \emph{algebraic}\label{p:alg-twist}. For example, for $\bG=\bZ$ and $f\in\Aut_{\bZ}\paren{A}$, we have the homomorphism $n\mapsto f^{n}$ and corresponding twisting system $\brac{f^{n} \mid n\in\bZ}$. 

In the case of a connected $\mathbb{N}$-graded algebra $A=\bigoplus A_{i}$ generated in degree 1, the fundamental defining relations for a Zhang twist are equivalent to the following (\cite[p.~284]{Zhang}).  A twisting system in this case is a set $\tau=\{ \tau_{n} \mid n\in \mathbb{N} \}$ of graded linear isomorphisms satisfying
\[ \tau_{m}(ab)=\tau_{m}(a)\tau_{m+n}\tau_{n}^{-1}(b) \]
for $a\in A_{n}$.  Notice that if $\tau_{m+n}\tau_{n}^{-1}$ were equal to $\tau_{m}$, we would have that $\tau_{m}$ was an algebra map (and in fact then an automorphism).  So this formulation makes visible another way in which twisting systems are close to, but more general than, algebra automorphisms.

As with other notions of twisting, such as twisting by automorphisms or by group 2-cocycles, given a twisting system $\tau$ for $A$, we can form a new algebra $A^{\tau}$, whose underlying $k$-vector space is the same as that of $A$ but whose multiplication is twisted by $\tau$.

\begin{definition}[{\cite[Definition and Proposition 2.3]{Zhang}}] Let $A$ be a $\bG$-graded $k$-algebra and $\tau=\brac{\tau_{g} \mid g\in\bG}$ be a twisting system for $A$. One defines the \emph{twisted algebra of $A$ by $\tau$}, denoted $A^{\tau}$, as the triple $\paren{\oplus_{g}A_{g},\star,1_{\tau}}$ where $\star$ is an associative, graded multiplication given by
  \[ y\star z\defeq y\tau_{h}\paren{z} \]
  for $y\in A_{h}$ and $z\in A_{l}$ and where $1_{\tau}\defeq\tau_{e}^{-1}\paren{1_{A}}$ is the identity in $A^{\tau}$.
\end{definition}

The following straightforward lemma concerns twisting systems for quadratic algebras and generalizes an observation of Zhang, made in examining Example~5.12 in \cite{Zhang}.

\begin{lemma} Let $A$ be a connected $\bN$-graded finitely generated algebra, generated in degree 1 and quadratic.  Let $\tau=\{ \tau_{m} \mid m\in \bN\}$ be a twisting system.  Suppose the graded linear isomorphism $\tau_{1}$ is additionally an algebra automorphism and set $\tau'=\{ \tau_{1}^{m} \mid m\in \bN \}$. Then we have an isomorphism of algebras $A^{\tau} \cong A^{\tau'}$. \qed
\end{lemma}

\begin{remark} Note that the condition on $\tau_{1}$ in the lemma is necessary, in the sense that $\tau'=\{ \tau_{1}^{m} \mid m\in \bN \}$ is a twisting system if and only if $\tau_{1}$ is an algebra automorphism, under the other assumptions of the lemma.
\end{remark}

Given a coordinate ring of a noncommutative $\pro{n}$ (that is, an algebra $A$ satisfying the conditions in Definition~\ref{Keeler}), Theorem 1.3 in \cite{Zhang} states that if $A$ is a Noetherian domain, then so is any twist of it. Clearly, the degree of an element is preserved by twisting, and hence any twist of $A$ is generated in degree 1 with Hilbert series $\paren{1-t}^{-n-1}$. From Theorem 5.11 of \cite{Zhang}, we have that any twist of $A$ will be AS-regular of the same dimension.  The class of connected graded $k$-algebras defined above is thus closed under Zhang twisting.  With this in mind, we make the following definition.

\begin{definition} Let $A$ and $B$ be $\bG$-graded $k$-algebras.  We say that $A$ and $B$ are twist-equivalent if there exists a twisting system $\tau$ for $A$ such that $B\iso A^{\tau}$ as graded $k$-algebras.
\end{definition}

The above description of Zhang twisting is rather concrete, which has advantages when one wants to explicitly calculate specific twists.  However, even demonstrating some elementary properties, such as showing that twist-equivalence is actually an equivalence relation, can be cumbersome when using this definition.

Along with the definition, the other key insight of Zhang was that twist-equivalence of $A$ and $B$ is equivalent to the existence of a well-behaved functor between certain corresponding categories of graded modules for $A$ and $B$.  From this, it is easy to see that twisting gives an equivalence relation.  We briefly outline the results of Zhang in this direction, taken from \cite{Zhang}, in order to motivate the definition of the groupoid we wish to study.

Assume that $A$ is a connected $\integ$-graded $k$-algebra.  In Definition~\ref{Keeler}, we assume a lot more about algebras that are associated to noncommutative $\mathbb{P}^{n}$s, but even this is already reasonably strong.

Let $\GR A$ denote the category of $\integ$-graded $A$-modules with morphisms being graded morphisms of any degree - that is, $\Hom_{\GR A}(M,N)=\bigoplus_{r\in \integ} \Hom_{\Gr A}(M,N(r))$ (note that Zhang writes $\underline{\Hom}(M,N)$).

We also note that $\Hom_{\GR A}(M,M)$, denoted $\Gamma(M)$ by Zhang, is important in a key theorem of Zhang (\cite[Theorem 3.3]{Zhang}), most notably in the case $M=A_{A}$.  We define $\GrEnd(A_{A}) \defeq \Hom_{\GR A}(A_{A},A_{A})$, the graded endomorphism algebra of $A_{A}$.  By the proof of \cite[Theorem 3.4]{Zhang}, with the above assumptions, $\GrEnd(A_{A})\iso A$; that is, by analogy with the ungraded theory, $A$ is isomorphic to its graded endomorphism algebra.

We say that $A$ and $B$ are graded Morita equivalent\footnote{Unfortunately, this terminology is used differently by different authors. Our usage follows Zhang's and is justified by the following proposition.} if $\GR A$ is equivalent to $\GR B$ by a functor $F$ which induces a map of graded rings on $\Hom$ spaces, i.e.\ $F(\Hom_{\GR A}(M,N)_{r})\subseteq \Hom_{\GR B}(FM,FN)_{r}$.  For connected graded algebras, such as those we consider, we have the following equivalent characterisations.

\begin{proposition}\label{p:categorical-graded-Morita-eq} Let $A$ and $B$ be coordinate rings of noncommutative $\pro{n}$s, so that $\paren{\QGr A,\pi A, (1)}$ and $\paren{\QGr B,\pi B, (1)}$ are noncommutative $\pro{n}$s.  The following are equivalent:
  \begin{enumerate}[label=\textup{(\roman*)}]
  \item $A\iso \GrEnd(P_{B})$ as graded algebras for $P_{B}$ a graded progenerator of $\Gr B$;
  \item $\GR A$ is equivalent to $\GR B$ via a functor preserving degrees of morphisms;
  \item\label{Gr-equiv-gr-functor} $\Gr A$ is equivalent to $\Gr B$ via a graded functor, i.e.\ one that commutes with the shift functors;
  \item $A$ is isomorphic to $B$ as graded algebras.
  \end{enumerate}
\end{proposition}

\begin{proof} The equivalence of the first three conditions holds more generally (not just for connected graded algebras) and details may be found in \cite[Theorem~5.4]{GordonGreen} and \cite[Theorem~2.3.8]{Hazrat}.

The equivalence of all of these with the fourth for coordinate rings of noncommutative $\pro{n}$s in particular follows from noting that for those algebras the only indecomposable graded projective right modules are shifts of the right regular module (cf. \cite[Theorem 3.5]{Zhang}), as they are connected graded. 
\end{proof}

Note that as a consequence of this proposition, we may replace \ref{Gr-equiv-gr-functor} by the stronger statement

\begin{enumerate}[label=\textup{(\roman*)$'$}]
  \setcounter{enumi}{2}
\item  $\Gr A$ is isomorphic to $\Gr B$ via a graded functor.
  \end{enumerate}

In fact, the strength of the assumptions on the algebras we are considering allows us to prove the following.

\begin{proposition} Let $F\colon \Gr A \to \Gr B$ be a graded Morita equivalence.  Then, up to shift, $F$ is isomorphic to the pushforward $f_{\ast}$ of a graded isomorphism of algebras $f\colon B\to A$.
\end{proposition}

\begin{proof} By graded Morita theory (see e.g. \cite[Theorem 2.3.7]{Hazrat}), $F$ is represented by a graded progenerator; that is, $F\iso \Hom_{\Gr A}(P,-)$ and $\GrEnd(P)\iso \GrEnd(B)\iso B$.  Again using that since $A$ is connected graded, the only indecomposable graded projective right modules are shifts of the right regular module, hence $P\iso A(\alpha)$ for some $\alpha \in \bZ$.  Hence, by shifting, we may assume that $P\iso A$.

We deduce that there exists an isomorphism $f\colon B\to A$ such that $F\iso \Hom_{\Gr A}(A,-)\iso f_{\ast}$ as required. The latter follows since the functor $-\otimes_{B} {}_{B}A_{A}$ is a right adjoint for both functors, where the bimodule structure on ${}_{B}A_{A}$ is that induced by $f$.
\end{proof}

As previously, we say $A$ and $B$ are twist-equivalent if $B$ is isomorphic to a Zhang twist of $A$.

\begin{proposition}\label{p:categorical-twist-eq} Let $A$ and $B$ be coordinate rings of noncommutative $\pro{n}$s, so that $\paren{\QGr A,\pi A, (1)}$ and $\paren{\QGr B,\pi B, (1)}$ are noncommutative $\pro{n}$s.  The following are equivalent:
  \begin{enumerate}[label=\textup{(\roman*)}]
  \item $A$ and $B$ are twist-equivalent, via a twisting system $\tau$.
  \item $\Gr A$ is isomorphic to $\Gr B$.
  \item $\QGr A$ is equivalent to $\QGr B$ via a functor $\mathcal{F}$ and
  \begin{equation}\label{SSS}\tag{SSS} \mathcal{F}((\pi A)(m))\iso (\pi B)(m)\ \text{for all}\ m \end{equation}
  (the ``preserves shifts of the structure sheaf'' condition).
  \end{enumerate}
\end{proposition}

\begin{proof} This follows by combining results of \cite[\S 3]{Zhang} and \cite{ArtinZhang}.  The key assumption is that the algebras considered are AS-Gorenstein and as such
\[ \bigoplus_{m\in \bN} \Hom_{\QGr A}(\pi A,(\pi A)(m)) \iso A \] which very much need not be the case in general.
\end{proof}

As noted above, Artin and Zhang proved that the graded endomorphism algebra of $\pi A$ in $\QGr A$,
  \[ \GrEnd(\pi A)\defeq \bigoplus_{m\in \bN} \Hom_{\QGr A}(\pi A,(\pi A)(m)) \]
is isomorphic as a graded $k$-algebra to $A$.  Hence, in our setting the algebra $A$ can be recovered from the data $(\QGr A, \pi A, (1))$. 

We thus take condition (iii) in Proposition~\ref{p:categorical-twist-eq}, $\QGr A$ being equivalent to $\QGr B$ and \eqref{SSS} holding, as the most appropriate formulation of the isomorphism of two noncommutative projective spaces. 

However it is useful to phrase certain statements in terms of properties of the algebras that are coordinate rings of noncommutative $\pro{n}$s, as we have being doing, as not all of the conditions of Definition~\ref{Keeler} have a convenient expression internal to $\QGr A$.

Noting the above remarks on graded endomorphism algebras, for completeness we may add additional items to our list of equivalent conditions in Proposition~\ref{p:categorical-graded-Morita-eq}:

\begin{enumerate}[label=\textup{(\roman*)}]
  \setcounter{enumi}{4}
\item\label{QGr-equiv-gr-functor} $\QGr A$ is equivalent to $\QGr B$ via a graded functor;
\item $A$ and $B$ are twist-equivalent via the identity twisting system.
  \end{enumerate}

\subsection{Geometric algebras}

Our approach will tie algebraic properties (being a twist of an algebra) to geometric ones, and to do this we will need to restrict the algebras we consider to those that are geometric, in the following sense due to Mori.

Let $V$ be a finite-dimensional vector space and $I$ a vector subspace of $V^{\otimes r}$. Denote by $\mathcal{V}(I)$ the projective subscheme of $\mathbb{P}(V^{\ast})^{r}$ determined by $I$, thinking of elements of $I$ as multilinear functions on $(V^{\ast})^{r}$.

\begin{definition}[{\cite[Definition~4.3]{Mori}}]\label{d:geometric} A quadratic algebra $A=T(V)/R$ is called geometric if there is a triple $(E,\sigma,\cL)$ where $j\colon E \to \mathbb{P}(V^{\ast})$ is an embedding of $E$ as a closed $k$-subscheme, $\cL$ is the associated line bundle and $\sigma$ is a $k$-automorphism of $E$ such that
  \begin{enumerate}
  \item $\Gamma_{2}=\mathcal{V}(R)\subset \mathbb{P}(V^{\ast})\times \mathbb{P}(V^{\ast})$ is the graph of $E$ under $\sigma$, and
  \item setting $\mathcal{L}=j^{\ast}\cO_{\mathbb{P}(V^{\ast})}(1)$, the map
    \[ \mu\colon H^{0}(E,\mathcal{L})\otimes H^{0}(E,\mathcal{L})\to H^{0}(E,\mathcal{L}\otimes_{\cO_{E}} \sigma^{\ast}\mathcal{L}) \] defined by $v\otimes w\mapsto v\otimes (w\circ \sigma)$ has $\ker \mu=R$, with the identification
    \[ H^{0}(E,\mathcal{L})=H^{0}(\mathbb{P}(V^{\ast}),\cO(\mathbb{P}(V^{\ast})(1))=V \] as $k$-vector spaces.
  \end{enumerate}
  If $A$ is geometric as above, we associate the triple $(E,\sigma, \cL)$ to both $A$ and the triple $(\QGr A, \pi A, (1))$; we will make this association more formal later.
\end{definition}

  \begin{theorem}[{\cite[Theorem~4.7]{Mori}}]\label{t:Mori-full} Let $A=T(V)/I$ and $A'=T(V)/I'$ be graded algebras finitely generated in degree 1 over $k$.
    \begin{enumerate}
    \item If $A$ is geometric with associated triple $(E,\sigma,\cL)$ and $\Gr A\iso \Gr A'$ then $A'$ is also geometric with associated triple $(E',\sigma',\cL')$ and there is a sequence of automorphisms $\{ \tau_{n}^{\ast} \}$ of $\mathbb{P}(V^{\ast})$, each of which sends $E$ isomorphically onto $E'$ such that $(\tau_{n+1}^{\ast}|_{E})\sigma = \sigma'(\tau_{n}^{\ast}|_{E})$ for every $n\in \integ$.
    \item Conversely if $A$ and $A'$ are geometric, with associated triples $(E,\sigma,\cL)$ and $(E',\sigma',\cL')$ respectively, and there is a sequence of automorphisms $\{ \tau_{n}^{\ast} \}$ of $\mathbb{P}(V^{\ast})$, each of which sends $E$ isomorphically onto $E'$ such that $(\tau_{n+1}^{\ast}|_{E})\sigma = \sigma'(\tau_{n}^{\ast}|_{E})$ for every $n\in \integ$, then $\Gr A \iso \Gr A'$.
      \end{enumerate}
\end{theorem}

  \noindent Diagrammatically, writing $\tau_{n}^{E}=(\tau_{n}^{\ast})|_{E}$,
  \begin{center}
\begin{tikzpicture}[node distance=1.5cm,on grid]
\node (a) at (0,0) {$E$};
\node (b) [below=of a] {$E'$};
\node (c) [right=of a] {$E$};
\node (d) [below=of c] {$E'$};
\node (e) [right=of c] {$E$};
\node (f) [below=of e] {$E'$};
\node (g) [right=of e] {$E$};
\node (h) [below=of g] {$E'$};
\node (i) [right=of g] {$E$};
\node (j) [below=of i] {$E'$};
\node (k) [right=of i] {};
\node (l) [below=of k] {};

\draw[->] (a) -- node [left] {$\tau_{0}^{E}$} (b);

\draw[->] (a) -- node [above] {$\sigma$} (c);
\draw[->] (b) -- node [below] {$\sigma'$} (d);
\draw[->] (c) -- node [left] {$\tau_{1}^{E}$} (d);

\draw[->] (c) -- node [above] {$\sigma$} (e);
\draw[->] (d) -- node [below] {$\sigma'$} (f);
\draw[->] (e) -- node [left] {$\tau_{2}^{E}$} (f);

\draw[->] (e) -- node [above] {$\sigma$} (g);
\draw[->] (f) -- node [below] {$\sigma'$} (h);
\draw[->] (g) -- node [left] {$\tau_{3}^{E}$} (h);

\draw[->] (g) -- node [above] {$\sigma$} (i);
\draw[->] (h) -- node [below] {$\sigma'$} (j);
\draw[->] (i) -- node [left] {$\tau_{4}^{E}$} (j);

\draw[->] (i) -- node [above] {$\sigma$} (k);
\draw[->] (j) -- node [below] {$\sigma'$} (l);

\end{tikzpicture}
\end{center}

  We note that if $A$, $A'$ are $\bN$-graded, the index $n$ may be taken to range over $\bN$ rather than $\bZ$, with no change to the conclusions.

As noted in Remark 4.9 of \cite{Mori}, one obtains from this result the well-known statement that, with notation as above, $A\iso A'$ as graded algebras if and only if there is an isomorphism $\tau\colon E\to E'$ which extends to an isomorphism $\bar{\tau} \colon \mathbb{P}(V^{\ast})\to \mathbb{P}((V')^{\ast})$ such that the diagram 
  \begin{center}
\begin{tikzpicture}[node distance=1.5cm,on grid]
\node (a) at (0,0) {$E$};
\node (b) [below=of a] {$E'$};
\node (c) [right=of a] {$E$};
\node (d) [below=of c] {$E'$};

\draw[->] (a) -- node [left] {$\tau$} (b);

\draw[->] (a) -- node [above] {$\sigma$} (c);
\draw[->] (b) -- node [below] {$\sigma'$} (d);
\draw[->] (c) -- node [left] {$\tau$} (d);

\end{tikzpicture}
\end{center}
  commutes.
  
We will insert the adjective ``geometric'' into our terminology for noncommutative projective spaces in the natural way, referring to $(\QGr A,\pi A, (1))$ as a geometric noncommutative $\pro{n}$ if $(\QGr A,\pi A, (1))$ is a noncommutative $\pro{n}$ for a geometric algebra $A$, and similarly refer to such an $A$ as a coordinate ring for a geometric noncommutative $\pro{n}$.

To conclude this section, we generalise a proposition of Artin--Tate--van den Bergh.  Specifically, in \S 8 of \cite{ATV2} the authors define the notion of twisting a 3-dimensional regular $\bZ$-graded algebra $A$ by a graded $A$-automorphism. This twisting process was later generalized by Zhang to that described above. They then define an action of graded $A$-automorphisms on the set of so-called regular triples and demonstrate a correspondence between twists of a given algebra and images of its regular triple under the action of $A$-automorphisms.  

In this proposition, we are able to remove the assumption on the dimension at the expense of requiring geometricity. We will also drop the term ``regular''. 

\begin{proposition}\label{p:twist-of-regular-triple} Let $A$ be a homogeneous coordinate ring for a geometric noncommutative $\pro{n}$ and let $(E,\sigma,\mathcal{L})$ be its associated triple.  Let $\tau$ be a twisting system for $A$.

By dualization and projectivization, there is an automorphism $\bar{\tau_{1}}^{\ast}$ of $\mathbb{P}(A_{1}^{\ast})$ induced by $\tau_{1}$ and $\bar{\tau_{1}}^{\ast}(E)=E$. Denote by $\tau_{1}^{E}$ the induced automorphism of $E$ and let $\sT_{\tau_{1}^{E}\sigma}$ be the triple $\paren{E,\tau_{1}^{E}\sigma,\cL}$. Then $A^{\tau}$ is the algebra determined (up to isomorphism) by $\sT_{\tau_{1}^{E}\sigma}$.
\end{proposition}

\begin{proof} By Theorem~\ref{t:Mori-full}, $A^{\tau}$ is geometric and has an associated triple $(E',\sigma',\mathcal{L}')$.  Now since $A^{\tau}$ is a twist of $A$, there are graded linear isomorphisms $\phi_{m}\colon A^{\tau}\to A$ such that $\phi_{m}(a\ast_{A^{\tau}} b)=\phi_{m}(a)\ast_{A}\phi_{m+l}(b)$ for $a\in A^{\tau}_{l}$, where $\ast_{A^{\tau}}$ and $\ast_{A}$ are the multiplications in $A^{\tau}$ and $A$ respectively.  These graded linear isomorphisms are obtained via \cite[Proposition~2.8(1)]{Zhang} and a careful examination of the proof of that result - which is stated for an algebra $B$ \emph{isomorphic to} a twist of $A$ - shows that if $B$ is \emph{equal} to a twist of $A$, i.e.\ we take $B=A^{\tau}$ in Zhang's proposition, then in fact $\phi_{m}=\tau_{m}$.  (In general, $\phi_{m}=\tau_{m}\circ f$ for $f\colon B\to A^{\tau}$.)

  Consequently, in Mori's theorem (\ref{t:Mori-full}), the key to the proof of the first part is the dualization of the maps $\tau_{m}|_{A_{1}}$, namely $\bar{\tau_{m}}^{\ast}\colon \mathbb{P}(A_{1}^{\ast})\to \mathbb{P}(A_{1}^{\ast})$.  (Since the $\tau_{m}|_{A_{1}}$ are isomorphisms, it follows that the maps $(\tau_{m}|_{A_{1}})^{\ast}$ are also isomorphisms and hence, being injective, they descend to the projectivization of $A_{1}^{\ast}$.)

  Then Mori shows that $E'=(\bar{\tau_{0}}^{\ast}|_{E})(E)$ and $\sigma'=(\bar{\tau_{1}}^{\ast}|_{E}) \circ \sigma \circ (\bar{\tau_{0}}^{\ast}|_{E})^{-1}$.  As above, set $\tau_{1}^{E}=\bar{\tau_{1}}^{\ast}|_{E}$.

  In principle, $\tau_{0}$ can be any graded $k$-linear automorphism of $A$ and then these claims would be best possible.  However, by \cite[Proposition~2.4]{Zhang}, we may if necessary - with an acceptable loss of generality - replace $\tau$ by a twisting system $\tau'$ for which $\tau'_{0}=\id$, and have $A^{\tau}\iso A^{\tau'}$.

  Doing this, and bearing in mind that we have potentially introduced a ``hidden'' isomorphism in the course of doing so, we conclude that the triples $(E',\sigma',\cL')$ and $(E,\tau_{1}^{E}\circ \sigma,\cL)$, with $\tau_{1}^{E}\in \Aut E$, yield isomorphic algebras, one of these being $A^{\tau}$.
  \end{proof}

\begin{remark} We note that \cite[Proposition~2.6]{ItabaMatsuno} proves the same result in the special case of $\tau$ being algebraic.
  \end{remark}

\begin{remark}\label{r:stablisation} Mori's theorem also tells us how to construct the full twisting system.  For as in \cite[Remark~4.8]{Mori}, the condition $(\bar{\tau_{n+1}}^{\ast}|_{E})\sigma=\sigma'(\bar{\tau_{n}}^{\ast}|_{E})$ implies that given $\bar{\tau_{1}}^{\ast}$, we may inductively define
  \[ \bar{\tau_{n}}^{\ast}\defeq ((\sigma')^{n}\sigma^{-n})^{\ast}=((\tau_{1}^{E}\sigma)^{n}\sigma^{-n})^{\ast} \]
  and dualize and extend these from automorphisms of $\mathbb{P}(A_{1}^{\ast})$ to automorphisms of $T(A_{1})$, which will define a twisting system $\tau$ on $A$.

  We note two consequences.  The first is that the twisting system $\tau$ obtained is completely determined by $\tau_{1}$, corresponding to $A$ being generated in degree 1.

  Secondly, if $\tau_{1}^{E}$ commutes with $\sigma$ then the twisting system $\tau$ is algebraic: $\bar{\tau_{n}}^{\ast}=(\tau_{1}^{E})^{n}$ and hence $\tau_{n}=\tau_{1}^{n}$ for all $n$.  As in \cite[Proposition~8.8]{ATV2} this is precisely the situation of $A^{\tau}$ being a twist by an automorphism.  Conversely, non-algebraic twists arise when we can find some $\tau_{1}^{E}$ that does not commute with $\sigma$ in $\Aut(E)$.
\end{remark}

Let $A$ and $B$ be coordinate rings of noncommutative $\pro{n}$s, so that $\paren{\QGr A,\pi A, (1)}$ and $\paren{\QGr B,\pi B, (1)}$ are noncommutative $\pro{n}$s.  Let $\cF\colon \QGr A \to \QGr B$ be an equivalence satisfying the condition \eqref{SSS}, fixing isomorphisms $\cF((\pi A)(m))\to (\pi B)(m)$ for all $m$.  Then \cite[Theorem~3.3]{Zhang} provides a construction of a family of maps $\phi=\{ \phi_{n}\}$ satisfying $\phi_{m}(ab)=\phi_{m}(a)\phi_{m+n}(b)$ for all $a\in B_{m}$ and $b\in B_{n}$. From this, define $\tau=\{ \tau_{m}=\phi_{m}\phi_{0}^{-1}\}$, a twisting system for $A$ (\cite[Proposition~2.8]{Zhang}); then $B\iso A^{\tau}$. 

\begin{definition}\label{d:geom-sys} With notation as in the previous paragraph, let $\bar{\tau}=\{ \bar{\tau_{n}}^{\ast}\colon \mathbb{P}(A_{1}^{\ast})\to \mathbb{P}(A_{1}^{\ast}) \}$ be the family of isomorphisms obtained from $\tau$ by dualization and projectivisation.  We will call $\bar{\tau}$ a \emph{geometric system} associated to $\cF$.
\end{definition}

The geometric system depends on the choice of isomorphisms in \eqref{SSS} if two choices of isomorphisms yield $\tau$ and $\tau'$ respectively, then both $A^{\tau}$ and $A^{\tau'}$ are isomorphic to $B$ and hence each other.  In what follows, we will only wish to work up to isomorphism of algebras so this ambiguity does not trouble us.

\section{Definition of $\ncpn$ for geometric noncommutative $\pro{n}$s}\label{s:ncpn}

We now define our main object of study, the groupoid $\mathscr{NC}\paren{\pro{n}}$:

\begin{definition}

Let $\mathscr{NC}\paren{\pro{n}}$ be the category whose objects are the triples $\paren{\QGr A, \pi A, (1)}$ where $A$ is a geometric noncommutative $\pro{n}$ and whose morphisms are pairs $(\cF,t)$ with $\cF:\QGr A\to\QGr B$ an equivalence of categories and $t=\{t_{m}\}$ a family of isomorphisms $t_{m}\colon\mathcal{F}\paren{\pi A(m)}\to\pi B(m)$ for all $m\in\bN$.
\end{definition}

By Proposition~\ref{p:categorical-twist-eq}, the groupoid $\ncpn$ partitions into connected components corresponding precisely to all of the Zhang twists of any given algebra $A$ whose triple $\paren{\QGr A, \pi A, (1)}$ is in that component. For example, for $n\geq 2$, $\mathcal{O}(\pro{n})=k\sqbrac{x_{0},x_{1},\ldots,x_{n}}$ and $\mathcal{O}_{q}(\pro{n}_{k})$ correspond to triples in different connected components. Henceforth, we call the component containing $(\QGr\cO(\pro{n}), \pi\cO(\pro{n}), (1))$ the \emph{commutative component}.

Our aim is to attempt to understand these connected components, to systematically introduce the study of generalized automorphisms into this strand of noncommutative geometry, in keeping with the classical literature of Artin, Stafford, Tate, Van den Bergh, Zhang et al. and the more recent work of Pym (\cite{Pym}).

We now adopt the calligraphic font $\mathcal{A}\defeq \paren{\QGr A, \pi A,(1)}$ for objects of $\ncpn$. Given an object $\mathcal{A}$, we wish to focus attention on the morphisms to $\mathcal{A}$. The connected components of $\ncpn$ do not in general have a terminal object but, for a given $\mathcal{A}$, we can form the slice category $\ncpn\slash\mathcal{A}$, whose terminal object is the morphism $\text{id}_{\mathcal{A}}:\mathcal{A}\rightarrow\mathcal{A}$. The objects of $\ncpn\slash\mathcal{A}$ are maps $\mathcal{B}\rightarrow\mathcal{A}$ and the morphisms are the triangles:
\vspace{-1em}
\begin{center}
\begin{tikzpicture}[node distance=1cm,on grid]
\node (B) at (0,0) {$\mathcal{B}$};
\node (C) [right=of B] {$\mathcal{C}$};
\node (A) [below=of B,xshift=0.5cm,yshift=0cm] {$\mathcal{A}$};
\draw[->] (B) -- node [above] {$f$} (C);
\draw[->] (B) -- node [above] {} (A);
\draw[->] (C) -- node [above] {} (A);
\end{tikzpicture}
\end{center}
where $f$ is a morphism in $\ncpn$.

The slice category $\ncpn/\mathcal{A}$ comes with a forgetful functor $\Phi$ to $\ncpn$, given on objects by taking the domain, $\Phi(\mathcal{B}\to \mathcal{A})=\mathcal{B}$ and on morphisms by taking the map $f$ in the triangle above.  In our situation, since $\ncpn$ is a groupoid, one can easily show that this functor is full.  However, it is not faithful; $\ncpn$ and its slices are closely related but do not hold identical information.  Part of the reason for this is that ``unique'' is a very strong statement in this set-up whereas ``unique up to isomorphism'' is very weak.

In order to study the groupoid $\ncpn$ and its slices, we will introduce a groupoid whose objects are triples and a functor from $\ncpn$ to this groupoid.  We will then construct a functor from this groupoid to the categorical analogue of a set with a group action, a so-called action groupoid, and thereby identify the structure that generalises the action of an automorphism group acting on an algebra to include the ``action'' of (non-algebraic) Zhang twists.

Throughout this section, $\cA$ is assumed to be a geometric noncommutative $\pro{n}$.

\begin{definition}\label{d:gtpn} Let $\gtpn$ be the category whose objects are triples $\cT=(E,\sigma,\cL)$ such that there exists $\cA \in \ncpn$ with triple $\cT$, and whose morphisms are as follows: for $\cT=(E,\sigma,\cL)$, $\cT'=(E',\sigma',\cL')$ define
  \[ \Hom_{\gtpn}(\cT,\cT')=\{ \bar{\tau}=\{ \tau_{i} \mid i\in \bN \} \mid \tau_{i}\in \Aut(\pro{n}), \tau_{i}^{E}\defeq \tau_{i}|_{E}\colon E\to E'\ \text{is an isomorphism and}\ \tau_{i+1}^{E}\sigma=\sigma'\tau_{i}^{E}\}. \]
\end{definition}

Composition is given by $\bar{\tau}'\circ \bar{\tau}=\{ \tau_{i}'\circ \tau_{i} \}$.  That is, morphisms in $\gtpn$ precisely correspond to the geometric systems of Definition~\ref{d:geom-sys} (i.e.\ the families of automorphisms associated to twists in Mori's theorem) and each such morphism gives rise to a commuting diagram
  \begin{center}
\begin{tikzpicture}[node distance=1.5cm,on grid]
\node (a) at (0,0) {$E$};
\node (b) [below=of a] {$E'$};
\node (c) [right=of a] {$E$};
\node (d) [below=of c] {$E'$};
\node (e) [right=of c] {$E$};
\node (f) [below=of e] {$E'$};
\node (g) [right=of e] {$E$};
\node (h) [below=of g] {$E'$};
\node (i) [right=of g] {$E$};
\node (j) [below=of i] {$E'$};
\node (k) [right=of i] {};
\node (l) [below=of k] {};

\draw[->] (a) -- node [left] {$\tau_{0}^{E}$} (b);

\draw[->] (a) -- node [above] {$\sigma$} (c);
\draw[->] (b) -- node [below] {$\sigma'$} (d);
\draw[->] (c) -- node [left] {$\tau_{1}^{E}$} (d);

\draw[->] (c) -- node [above] {$\sigma$} (e);
\draw[->] (d) -- node [below] {$\sigma'$} (f);
\draw[->] (e) -- node [left] {$\tau_{2}^{E}$} (f);

\draw[->] (e) -- node [above] {$\sigma$} (g);
\draw[->] (f) -- node [below] {$\sigma'$} (h);
\draw[->] (g) -- node [left] {$\tau_{3}^{E}$} (h);

\draw[->] (g) -- node [above] {$\sigma$} (i);
\draw[->] (h) -- node [below] {$\sigma'$} (j);
\draw[->] (i) -- node [left] {$\tau_{4}^{E}$} (j);

\draw[->] (i) -- node [above] {$\sigma$} (k);
\draw[->] (j) -- node [below] {$\sigma'$} (l);

\end{tikzpicture}
\end{center}
Composition of geometric systems $\bar{\tau}$ is the natural one such that on the associated diagrams, composition is given by stacking diagrams vertically and composing the constituent maps.  It is important to note, however, that $\tau_{i}^{E}$ may not have a unique extension to $\pro{n}$ and so the diagram above may contain less information than the system $\bar{\tau}$.

It is straightforward to verify that $\gtpn$ is a groupoid, with $\bar{\tau}^{-1}=\{ \tau_{i}^{-1} \}$.  We will refer to $\gtpn$ as the groupoid of geometric triples of noncommutative $\pro{n}$s.

Let $\cT=(E,\sigma,\cL)$ and $\cT'=(E',\sigma',\cL')$ be objects in $\gtpn$.  We note that the existence of a morphism $\cT \to \cT'$ in $\gtpn$ implies the following (strong) relationship between the point schemes $E$ and $E'$: namely, there exists an automorphism $\tau_{0}\in \Aut \pro{n}$ such that $\tau_{0}$ maps $E$ isomorphically to $E'$.  That is, $E'$ and $E$ are not only isomorphic as abstract schemes, but they are isomorphic via an automorphism of the ambient projective space; expressed in terms of the associated line bundles, this means that $\cL'\iso \tau_{0}^{*}\cL$.

\begin{theorem} There is a functor $\cT\colon \ncpn \to \gtpn$ defined on objects by $\cA \mapsto \cT(\cA)$ (that is, $\cT$ sends each geometric noncommutative $\pro{n}$ $\cA$ to its associated geometric triple) and on morphisms by $(\cF,t) \mapsto \bar{\tau}$ where $\bar{\tau}$ is the geometric system associated to $\cF$ and $t$.
\end{theorem}

\begin{proof} The functor $\cT$ is well-defined by the discussion before Definition~\ref{d:geom-sys}, explaining the construction of the geometric system associated to the choice of $\cF$ and $t$.
  We need to show that the construction in \cite[Theorem~3.3]{Zhang} of the family $\phi=\{ \phi_{n} \}$ of graded linear isomorphisms associated to a twisting functor is (contravariantly) functorial.  Then since $\bar{\tau}$ is constructed from $\phi$ by dualization, which is also contravariant, we will conclude that the functor $\cT$ is (covariantly) functorial.

  Let $\cF \colon \QGr A \to \QGr B$ and $\cG \colon \QGr B \to \QGr C$ be morphisms in $\ncpn$.  Let $F\colon \Gr A \to \Gr B$ and $G\colon \Gr B \to \Gr C$ be the corresponding isomorphisms, noting that we have isomorphisms $t_{n}\colon F(A(n)) \to B(n)$ and $u_{n}\colon G(B(n)) \to C(n)$ for all $n$ (cf.\ \cite[Theorem~3.4]{Zhang}).  Hence there exist isomorphisms $v_{n}\defeq u_{n}\circ (Gt_{n})\colon (G\circ F)(A(n)) \to C(n)$.

  In order to match the notation of the proof of \cite[Theorem~3.3]{Zhang}, let us write $S_{n}$ for all shift functors acting on objects and morphisms in the graded module categories at hand.  Then, as in that proof, we have associated to $F$ the family of maps
  \begin{align*} & \phi_{n}\colon \bigoplus_{m} \Hom_{\Gr B}(S_{m}(B),B) \to \bigoplus_{m} \Hom_{\Gr A}(S_{m}(A),A), \\ & \phi_{n}(a)=(S_{n})^{-1}(F^{-1}(t_{n}^{-1}S_{n}(a)t_{m+n}))\ \text{for all}\ a\in \Hom_{\Gr B}(S_{m}(B),B).
    \end{align*}
  Let $\phi_{n}'$ (respectively $\phi_{n}''$) be the family of maps constructed in the same way from $G$ (respectively $G\circ F$).

  For $a\in \Hom_{\Gr C}(S_{n}(C),C)$, we have
  \begin{align*} \phi_{n}''(a) & = (S_{n})^{-1}((G\circ F)^{-1}(v_{n}^{-1}S_{n}(a)v_{m+n})) \\
    & = (S_{n}^{-1})(F^{-1}(G^{-1}((Gt_{n}^{-1}\circ u_{n}^{-1})S_{n}(a)(u_{m+n}\circ Gt_{m+n}^{-1})))) \\
    & = (S_{n}^{-1})(F^{-1}(t_{n}^{-1}G^{-1}(u_{n}^{-1}S_{n}(a)u_{m+n})t_{m+n}^{-1})) \\
    & = (S_{n}^{-1})(F^{-1}(t_{n}^{-1}(S_{n}(S_{n}^{-1}(G^{-1}(u_{n}^{-1}S_{n}(a)u_{m+n}))))t_{m+n}^{-1})) \\
    & = (S_{n}^{-1})F^{-1}(t_{n}^{-1}S_{n}(\phi_{n}'(a))t_{m+n}^{-1}) \\
    & = \phi_{n}(\phi_{n}'(a)) = (\phi_{n}\circ \phi_{n}')(a),
  \end{align*}
  as required.
\end{proof}

Note that, by the definition of $\gtpn$ (Definition~\ref{d:gtpn}), the functor $\cT\colon \ncpn \to \gtpn$ is full and surjective (not just essentially surjective) but it is not necessarily faithful. It is full because $\gtpn$ consists of only those triples arising as the triple of a geometric noncommutative $\pro{n}$.

The functor $\cT$ between the groupoids $\ncpn$ and $\gtpn$ induces in a natural way a functor between the slice groupoids $\ncpn/\cA$ and $\gtpn/\cT(\cA)$.  For it is straightforward to verify\footnote{These claims apply to any functor between two categories; they do not use any properties of the categories at hand.} that $\cT_{\cA}\colon \ncpn/\cA \to \gtpn/\cT(\cA)$ defined
\begin{itemize}
\item on objects of the slice category $\ncpn/\cA$ by $\cT_{\cA}(\cF\colon \cB \to \cA)=(\cT\cF\colon \cT\cB \to \cT\cA)$ and
\item on morphisms of the slice category, which are commuting triangles, by the image triangle under $\cT$
\end{itemize}
is a functor.  Since $\cT$ is full, $\cT_{\cA}$ is both surjective on objects and full.  Note that if $\cT$ were faithful, $\cT_{\cA}$ would be also.

The functor $\cT_{\cA}$ gives rises to a surjective function $t_{\cA}$ with domain the objects of $\ncpn/\cA$ and codomain the objects of $\gtpn/\cT(\cA)$.  (It is elementary to check that in a slice groupoid $\cG/X$, we have $|\Hom_{\cG/X}(f,f')|=1$ for all objects $f,f'$ of $\cG/X$.  This implies that to understand a slice groupoid, the principal task is to understand its objects, in contrast to the more common situation where attention is focused on understanding morphisms.)

Now the (connected) groupoid $\gtpn/\cT(\cA)$ has a natural distinguished object, namely the identity morphism $\bar{\id_{\cT(\cA)}}\colon \cT(\cA) \to \cT(\cA)$ given by $(\bar{\id_{\cT(\cA)}})_{i}=\id_{\pro{n}}$.  We can then consider the pre-image of this object under $t_{\cA}$: we set $K_{\cA}=t_{\cA}^{-1}(\bar{\id_{\cT(\cA)}})$.  Then $\cF\colon \cB \to \cA$ belongs to $K_{\cA}$ if and only if $\cT\cF\colon \cT\cB\to \cT\cA$ is equal to $\bar{\id_{\cT(\cA)}}$, which in particular implies that $\cT\cB=\cT\cA$.

In the groupoid $\gtpn$, we have a distinguished class of morphisms, as follows.  Let us say that $\bar{\tau}=\{ \tau_{n} \} \in \Hom_{\gtpn}(\cT,\cT')$ is \emph{constant} if $\tau_{i}=\tau_{j}$ for all $i,j$.  Let $\text{Con}(\cT,\cT')$ be the subset of $\Hom_{\gtpn}(\cT,\cT')$ of constant morphisms.  Note that $\bar{\id_{\cT}}$, $\bar{\id_{\cT}}_{i}=\id_{\pro{n}}$, is an example of a constant morphism, and that $\text{Con}(\cT,\cT')$ may be empty (if $\cT\neq \cT'$).

Then one sees that there is a subcategory $\sC(\pro{n})$ with the same objects as $\gtpn$ and morphism sets $\text{Con}(\cT,\cT')$, with composition being the restriction of composition in $\gtpn$.  Indeed, $\sC(\pro{n})$ is again a groupoid.

The significance of the constant morphisms is explained by the following proposition.

\begin{proposition}[{cf.\ \cite[Remark~4.9]{Mori}}]\label{p:const-morph-iso} Let $A$ and $B$ be coordinate rings of noncommutative $\pro{n}$s, so that $\cA=\paren{\QGr A,\pi A, (1)}$ and $\cB=\paren{\QGr B,\pi B, (1)}$ are noncommutative $\pro{n}$s. Let $\cT=\cT(\cA)$ and $\cT'=\cT(\cB)$ be the corresponding geometric triples.

  Then there exists $\bar{\tau}\in \text{Con}(\cT,\cT')$ if and only if $A$ is isomorphic to $B$, as graded algebras.
\end{proposition}

\begin{proof} This is a rewriting in our terminology of \cite[Remark~4.9]{Mori}, which was, as noted there, already well-known; the point here and in what follows is to make very explicit the relationship between isomorphism and twist-equivalence. The forwards implication is proved by seeing that the twisting system $\{ \theta_{n}\}$ such that $B\iso A^{\theta}$ arising from $\bar{\tau}$ is given by $\theta_{n}=(\tau_{0}^{-1}\tau_{n})^{\ast}$ (cf. \cite[Remark~4.8]{Mori}) so $\theta_{n}$ is the identity for all $n$ if $\bar{\tau}$ is constant.
  \end{proof}

An immediate corollary of this is that functors $\cF\colon \cB\to \cA$ that belong to $K_{\cA}$ are twists that yield isomorphic algebras.  For being in $K_{\cA}$, $\cT\cF=\bar{\id_{\cT(\cA)}}$ so that the geometric system associated to $\cF$ is constant.  As such, our focus from this point will be on understanding $\gtpn/\cT$.

Let $\cT=(E,\sigma,\cL)$, $\cT'=(E',\sigma',\cL')$ be triples in $\gtpn$.  Let $\bar{\tau}\colon \cT'\to \cT$ be a morphism in $\Hom_{\gtpn}(\cT',\cT)$ (or equivalently, an object in $\gtpn/\cT$), so $\bar{\tau}=\{ \tau_{i} \mid i\in \nat\}$.  Then let us form the constant morphism generated by $\tau_{0}$: define $\bar{\tau_{0}}$ to be the morphism $\bar{\tau_{0}}\colon \cT'\to \cT$ with $(\bar{\tau_{0}})_{i}=\tau_{0}$ for all $i$, so $\bar{\tau_{0}} \in \text{Con}(\cT',\cT)$.  As before, we write $\tau_{i}^{E}$ for $\tau_{i}|_{E}$.

Consider the triple $\cT''=(E,\tau_{0}^{E}(\tau_{1}^{E})^{-1}\sigma,\cL)$.  Then $\bar{\nu}\colon \cT''\to \cT$ defined by $\bar{\nu}_{i}=\tau_{i}\tau_{0}^{-1}$ is a morphism in $\gtpn$, since
\begin{align*} \nu_{i+1}^{E}(\tau_{0}^{E}(\tau_{1}^{E})^{-1})\sigma & = \tau_{i+1}^{E}(\tau_{0}^{E})^{-1}(\tau_{0}^{E}(\tau_{1}^{E})^{-1})\sigma \\
  & = \tau_{i+1}^{E}(\tau_{1}^{E})^{-1}\sigma \\
  & = \tau_{i+1}^{E}(\sigma'(\tau_{0}^{E})^{-1}) \\
  & = \sigma\tau_{i}^{E}(\tau_{0}^{E})^{-1} \\
  & = \sigma\nu_{i}^{E}
\end{align*}
and evidently has the property that $\bar{\nu}=\bar{\tau}\circ\bar{\tau_{0}}^{-1}$.  That is, the algebras represented by the triples $\cT'$ and $\cT''$ are isomorphic.  

This observation may be generalised to the following result, where we see that we can instead fix $\cT$ and the equivalent of $\tau_{0}$, we can let $\cT'$ be determined by these, and still obtain constant morphisms of triples.

\begin{proposition}\label{p:PGL-action-on-GT-slice-T} Fix $\cT=(E,\sigma,\cL)$. There is a homomorphism $\rho\colon \op{\Aut(\pro{n})} \to \Aut(\gtpn/\cT)$ sending $\tau\in \Aut(\pro{n})$ to the isomorphism of categories $\rho(\tau)\colon \gtpn/\cT \to \gtpn/\cT$ defined on objects as follows. Let $\bar{\nu}\colon \cT' \to \cT$ be an object of $\gtpn/\cT$ with $\cT'=(E',\sigma',\cL')$. Then define a constant geometric system (depending on $\tau$ and $\bar{\nu}$) by setting $\tau^{E'}=\tau|_{E'}$ and defining
    \begin{equation*}\label{eq:induced-morphism} \rho(\tau,\bar{\nu})\colon (E',\sigma',\cL') \to (\tau(E'),\tau^{E'}\circ \sigma' \circ (\tau^{E'})^{-1},((\tau^{E'})^{-1})^{*}\cL')
    \end{equation*}
by $\rho(\tau,\bar{\nu})_{i}=\tau$ for all $i$.  

Then set 
  \[ \rho(\tau)(\bar{\nu})=\bar{\nu} \circ \rho(\tau,\bar{\nu})^{-1} \]

On morphisms, $\rho(\tau)$ is given by
\begin{center}
  \begin{tikzpicture}[node distance=2.5cm,on grid]
    \begin{scope}[xshift=-5cm,yshift=-0.25cm]
      \node (B) at (0,0) {$(E',\sigma',\cL')$};
\node (C) [right=of B] {$(E'',\sigma'',\cL'')$};
\node (A) [below=of B,xshift=1.25cm,yshift=0.75cm] {$(E,\sigma,\cL)$};
\draw[->] (B) -- node [above] {$\bar{\mu}$} (C);
\draw[->] (B) -- node [left,xshift=-0.25cm] {$\bar{\nu}$} (A);
\draw[->] (C) -- node [right] {$\bar{\nu'}$} (A);
    \end{scope}
    \node (circ) at (-0.75,-1.25) {$\mapsto$};
    \begin{scope}[xshift=2.5cm,yshift=-1.25cm]
            \node (B) at (0,0) {$(E',\sigma',\cL')$};
\node (A) [below right=of B] {$(E,\sigma,\cL)$};
\node (C) [above right=of A] {$(E'',\sigma'',\cL'')$};
\node (D) [above left=of B] {$(\tau(E'),(\sigma')^{\tau^{E'}},((\tau^{E'})^{-1})^{*}\cL')$};
\node (E) [above right=of C] {$(\tau(E''),(\sigma'')^{\tau^{E''}},((\tau^{E''})^{-1})^{*}\cL'')$};
\draw[->,dashed] (B) -- node [above] {$\bar{\mu}$} (C);
\draw[->] (B) -- node [left,xshift=-0.25cm] {$\bar{\nu}$} (A);
\draw[->] (C) -- node [right] {$\bar{\nu'}$} (A);
\draw[->] (D) -- node [left,xshift=-0.25cm] {$\rho(\tau,\bar{\nu})^{-1}$} (B);
\draw[->] (D) -- node [above,yshift=0.2cm] {$\rho(\tau,\bar{\nu'})\circ \bar{\mu} \circ \rho(\tau,\bar{\nu})^{-1}$} (E);
\draw[->] (E) -- node [right] {$\rho(\tau,\bar{\nu'})^{-1}$} (C);
\end{scope}
\end{tikzpicture}
  \end{center}
  where $(\sigma')^{\tau^{E'}}\defeq \tau^{E'}\circ \sigma' \circ (\tau^{E'})^{-1}$.
\end{proposition}

\begin{proof}
  This is a straightforward verification.
\end{proof}

\begin{lemma}\label{l:rho-is-free} In the setting of the previous Proposition, we have
  \begin{enumerate}[label=(\alph*)]
  \item for all $\bar{\nu}\in \gtpn/\cT$,
    \[ \mathrm{Stab}_{\rho}(\bar{\nu})\defeq \{ \tau \in \Aut(\pro{n}) \mid \rho(\tau)(\bar{\nu})=\bar{\nu} \} \]
    is trivial; and
  \item the kernel of $\rho$ is trivial.
  \end{enumerate}
\end{lemma}

\begin{proof} From the definition of $\rho$, we see that we have $\rho(\tau)(\bar{\nu})=\bar{\nu}$ if and only if $\bar{\nu}=\bar{\nu}\circ \rho(\tau,\bar{\nu})^{-1}$, as objects in $\gtpn/\cT$.  Viewing them as morphisms in $\gtpn$, $\bar{\nu}$ is invertible, so this is equivalent to $\rho(\tau,\bar{\nu})=\id_{\cT'}$ as morphisms in $\gtpn$.  (Note that $\rho(\tau,\bar{\nu})$ is not an object in $\gtpn/\cT$.)

  Now $\id_{\cT'}$ is precisely the constant geometric system $\bar{\id_{\pro{n}}}$ obtained from $\id_{\pro{n}}\in \Aut(\pro{n})$, i.e. with $\bar{\id_{\pro{n}}}_{i}=\id_{\pro{n}}$ for all $i$, with the restrictions $\id_{\pro{n}}|_{E'}$ being equal to $\id_{E'}$.  Since $\rho(\tau,\bar{\nu})$ is the constant geometric system obtained from $\tau$, we deduce that $\rho(\tau,\bar{\nu})=\id_{\cT'}$ implies $\tau=\id_{\pro{n}}$.  Hence $\mathrm{Stab}_{\rho}(\bar{\nu})=\{ \id_{\pro{n}} \}$.

  It immediately follows that $\ker \rho = \bigcap_{\bar{\nu}} \mathrm{Stab}_{\rho}(\bar{\nu})$ is trivial also.
\end{proof}

\begin{proposition} Let $(\bar{\nu}\colon \cT'\to \cT)\in \gtpn/\cT$.  Define
  \[ \mathrm{Orb}_{\rho}(\bar{\nu})\defeq \{ \rho(\tau)(\bar{\nu}) \mid \tau \in \Aut(\pro{n}) \}. \]
  Then there is a bijection 
  \[ \phi \colon \mathrm{Orb}_{\rho}(\bar{\nu}) \to \bigcup_{\cT''\in \gtpn} \mathrm{Con}(\cT'',\cT'). \]
\end{proposition}

\begin{proof} Let $\rho(\tau)(\bar{\nu})\in \mathrm{Orb}_{\rho}(\bar{\nu})$.  By definition, we have $\rho(\tau)(\bar{\nu})=\bar{\nu}\circ \rho(\tau,\bar{\nu})^{-1}$.  We define
  \[ \phi(\rho(\tau)(\bar{\nu}))=\bar{\nu}^{-1}\circ \rho(\tau)(\bar{\nu})=\rho(\tau,\bar{\nu})^{-1}, \] noting that by construction $\rho(\tau,\bar{\nu})^{-1}$ is a constant geometric system whose codomain is $\cT'$, the domain of $\bar{\nu}$.

  Conversely, let us take a constant geometric system whose codomain is $\cT'$ and construct a suitable map to $\mathrm{Orb}_{\rho}(\bar{\nu})$.  Specifically, let $\psi\colon \bigcup_{\cT''\in \gtpn} \mathrm{Con}(\cT'',\cT') \to \mathrm{Orb}_{\rho}(\bar{\nu})$ be defined as follows.  For $\bar{\mu}\in \mathrm{Con}(\cT'',\cT')$, set $\psi(\bar{\mu})=\bar{\nu}\circ \bar{\mu}$.  We need to show that $\bar{\nu}\circ \bar{\mu}\in \mathrm{Orb}_{\rho}(\bar{\nu})$, which reduces to showing that there exists $\tau \in \Aut(\pro{n})$ such that $\bar{\mu}=\rho(\tau,\bar{\nu})^{-1}$.

  Now $\bar{\mu}$ is a constant geometric system, so is given by $\bar{\mu}=\{ \mu_{i} \}$ with $\mu_{i}\in \Aut(\pro{n})$ and furthermore $\mu_{i}=\mu_{j}$ for all $i,j$; let $\mu=\mu_{0}$ ($=\mu_{1}=\dotsc $). We must have that $\cT''=(E'',\sigma'',\cL'')$ satisfies $\mu(E'')=E'$, $\sigma''=\mu_{E''}^{-1}\circ \sigma' \circ \mu_{E''}=(\sigma')^{\mu_{E''}^{-1}}$ and $\cL''=\mu^{*}\cL'$, by the definition of a geometric system.

Hence set $\tau=\mu^{-1}$ so that $\tau_{E'}=(\mu_{E''})^{-1}$ and
  \[ \rho(\tau,\bar{\nu})\colon \cT' \to (\tau(E'),(\sigma')^{\tau_{E'}},(\tau_{E'}^{-1})^{*}\cL')=(E'',\sigma'',\cL''), \]
  with $(\rho(\tau,\bar{\nu})^{-1})_{i}=\tau^{-1}=\mu=(\bar{\mu}_{i})$ for all $i$. So $\bar{\mu}=\rho(\mu^{-1},\bar{\nu})^{-1}$ as required.
  
  It is straightforward to see that $\phi$ and $\psi$ are mutually inverse.
\end{proof}

We claim that in each orbit there is a canonical element whose domain is a geometric triple such that the point scheme in this triple is equal to (and not just isomorphic) to $E$, the point scheme in the chosen triple $\cT$.

\begin{lemma} Let $\mathrm{Orb}_{\rho}(\bar{\nu})$ be the $\rho$-orbit of $\bar{\nu}$ as defined previously.  Then there exists a unique element $\bar{o}\in \mathrm{Orb}_{\rho}(\bar{\nu})$ with $\bar{o}\colon (E,\sigma',\cL')\to (E,\sigma,\cL)$ such that $\bar{o}_{0}=\id_{\pro{n}}$.
\end{lemma}

\begin{proof} Let $\bar{\mu}$ be the constant geometric system with $\mu_{i}=\nu_{0}^{-1}$ for all $i$.  Then $\bar{\nu}\circ \bar{\mu}\in \mathrm{Orb}_{\rho}(\bar{\nu})$ and
  \[ (\bar{\nu}\circ \bar{\mu})_{0}=\nu_{0}\circ \mu_{0} = \nu_{0}\circ \nu_{0}^{-1}=\id_{\pro{n}} \]
  so that we may choose $\bar{o}=\bar{\nu}\circ \bar{\mu}$, giving existence.

  For uniqueness, if $\bar{\nu'}\in \mathrm{Orb}_{\rho}(\bar{\nu})$ is such that $\bar{\nu'}_{0}=\id_{\pro{n}}$ then there exists a constant geometric system $\bar{\mu}$ such that $\bar{\nu'}=\bar{\nu}\circ \bar{\mu}$.  Then $\nu_{0}\circ \mu_{0}=\id_{\pro{n}}$, so that $\mu_{0}=\nu_{0}^{-1}$.  Since $\bar{\mu}$ is constant, $\mu_{i}=\nu_{0}^{-1}$ for all $i$ and $\bar{\nu'}=\bar{o}$, hence uniqueness.
\end{proof}

Let us call the element $\bar{o}\in \mathrm{Orb}_{\rho}(\bar{\nu})$ the \emph{standard representative} of $\mathrm{Orb}_{\rho}(\bar{\nu})$.  Since orbits partition, it immediately follows that every geometric system $\bar{o'}$ with $\bar{o'}_{0}=\id_{\pro{n}}$ is the standard representative of exactly one orbit, $\mathrm{Orb}_{\rho}(\bar{o'})$.

Let $E$ be a closed subscheme of $\pro{n}$ and define the following two subgroups:
\begin{align*}
  \Aut(E \uparrow \pro{n}) & \defeq \{ \sigma \in \Aut E \mid (\exists \hat{\sigma}\in \Aut \pro{n})(\hat{\sigma}|_{E}=\sigma) \} \subseteq \Aut E, \\
  \Aut(\pro{n} \downarrow E) & \defeq \{ \rho \in \Aut \pro{n} \mid \rho(E)=E \} \subseteq \Aut \pro{n}.
\end{align*}
Restriction defines a surjective group homomorphism $\mathrm{Res}_{E}\colon \Aut(\pro{n} \downarrow E) \to \Aut(E \uparrow \pro{n})$.

Consider an orbit $\mathrm{Orb}_{\rho}(\bar{\nu})$ and let $\bar{o}\colon (E,\sigma',\cL')\to (E,\sigma,\cL)$ be its standard representative, i.e.\ the unique element with $\bar{o}_{0}=\id_{\pro{n}}$, as in the lemma.  We continue with our notation $o_{i}^{E}=o_{i}|_{E}$. Then, from the definition of geometric systems, we have that
$o_{r}^{E}\sigma'=\sigma o_{r-1}^{E}$ yields $o_{r}^{E}=\sigma^{r-1}(o_{1}^{E}\sigma^{-1})^{r-1}o_{1}^{E}\in \Aut(E \uparrow \pro{n})$
so $\bar{o}$ is \emph{generated in degree 1}, i.e.\ $o_{r}^{E}$ is a function of $o_{1}^{E}$ and $\sigma$, for all $r$.  Notice too that if $o_{1}^{E}$ commutes with $\sigma$, the formula for $o_{r}^{E}$ simplifies to $o_{r}^{E}=(o_{1}^{E})^{r}$.

Recalling that our aim is to parameterize twisting systems of an algebra, via geometric systems, by group-theoretic data associated to the point scheme, we now define a function as follows.

\begin{definition} Let $\cP=\Aut{\pro{n}}$ and
  \[ (\gtpn/\cT) \git \cP = \{ \mathrm{Orb}_{\rho}(\bar{\nu}) \mid \bar{\nu}\in \gtpn/\cT \} \]
  be the set of $\rho$-orbits.  Define a function $\Sigma\colon (\gtpn/\cT) \git \cP \to \Aut(E \uparrow \pro{n})\sigma$ by $\Sigma(\mathrm{Orb}_{\rho}(\bar{\nu}))=(o_{1}^{E})^{-1}\sigma$, for $o_{1}^{E}=\bar{o}_{1}|_{E}$, $\bar{o}$ the standard representative of $\mathrm{Orb}_{\rho}(\bar{\nu})$.
\end{definition}

It follows from the results of this section that this function is well-defined.

\emph{A priori}, this function discards information, but this is a necessary compromise.  Note that by defining this function on $\rho$-orbits, we are (as shown immediately prior to Proposition~\ref{p:PGL-action-on-GT-slice-T}) associating the group-theoretic datum of a coset element to an ``isomorphism class'' of twists; more precisely, the coset element is an invariant under taking twists that yield isomorphic algebras.  

The extent to which the function does or does not lose information relates to our ability to prove injectivity and/or surjectivity.  In specific examples, with more information about the geometry concerned, stronger results may be possible.

\begin{lemma}\label{l:ker-of-Sigma} Let $\mathrm{Orb}_{\rho}(\bar{\nu})\in (\gtpn / \cT) \git \cP$. There is a bijection between the sets
  \begin{enumerate}
  \item $\Sigma^{-1}(\Sigma(\mathrm{Orb}_{\rho}(\bar{\nu})))$, the pre-image under $\Sigma$ of $\Sigma(\mathrm{Orb}_{\rho}(\bar{\nu}))\in \Aut(E \uparrow \pro{n})\sigma$ and
  \item $(\ker \mathrm{Res}_{E})^{\nat}=\{ \{ k_{r} \mid r\in \nat \} \mid k_{r}\in \ker \mathrm{Res}_{E} \}$, the set of $\nat$-indexed sequences of elements of $\ker \mathrm{Res}_{E}$.
    \end{enumerate}
\end{lemma}

\begin{proof} Let $\mathrm{Orb}_{\rho}(\bar{\mu})$ and $\mathrm{Orb}_{\rho}(\bar{\nu})$ be orbits such that $\Sigma(\mathrm{Orb}_{\rho}(\bar{\mu}))=\Sigma(\mathrm{Orb}_{\rho}(\bar{\nu}))$.  Let $\bar{o}$ and $\bar{\pi}$ be the standard representatives of $\mathrm{Orb}_{\rho}(\bar{\mu})$ and $\mathrm{Orb}_{\rho}(\bar{\nu})$ respectively, so that
  \[ (o_{1}^{E})^{-1}\sigma=\Sigma(\mathrm{Orb}_{\rho}(\bar{\mu}))=\Sigma(\mathrm{Orb}_{\rho}(\bar{\nu}))=(\pi_{1}^{E})^{-1}\sigma \]
  Then $o_{1}^{E}=\pi_{1}^{E}$ and, since $\bar{o}$ and $\bar{\pi}$ are generated in degree 1, $o_{r}^{E}=\pi_{r}^{E}$ for all $r$.

  Let $K=\ker \mathrm{Res}_{E}$, so that $\mathrm{Res}_{E}^{-1}(o_{r}^{E})=o_{r}K$.  Since $\mathrm{Res}_{E}(\pi_{r})=\pi_{r}^{E}=o_{r}^{E}$, we must have $\pi_{r}\in o_{r}K$ and hence there exists $k_{r}\in K$ such that $\pi_{r}=o_{r}k_{r}$.  Notice that since $\bar{o}$ and $\bar{\pi}$ are standard, $o_{0}=\pi_{0}=\id_{\pro{n}}$ and so $k_{0}=\id_{\pro{n}}$.

  Denote by $\bar{ok}$ the geometric system $\bar{ok}=\{ o_{r}k_{r} \}$.  This is a geometric system, since $\bar{o}$ is and $\mathrm{Res}_{E}(o_{r}k_{r})=\mathrm{Res}_{E}(o_{r})\mathrm{Res}_{E}(k_{r})=\mathrm{Res}_{E}(o_{r})=o_{r}^{E}$.  Note that $\bar{ok}$ is standard in its orbit since $\bar{ok}_{0}=\id_{\pro{n}}$.  Then $\Sigma(\mathrm{Orb}_{\rho}(\bar{ok}))=(\mathrm{Res}_{E}(o_{1}k_{1}))^{-1}\sigma=(o_{1}^{E})^{-1}\sigma=\Sigma(\mathrm{Orb}_{\rho}(\bar{o}))$.

  Indeed, the same argument shows that if $\bar{k}=\{ k_{r} \}$ is any choice of $\nat$-indexed sequence of elements of $K$, then $\bar{k}$ is automatically a geometric system, with $k_{r}^{E}=\id_{E}$ for all $r$.  Moreover, for any geometric system $\bar{o}$, $\bar{ok}=\bar{o}\circ \bar{k}$ satisfies $\Sigma(\mathrm{Orb}_{\rho}(\bar{o}))=\Sigma(\mathrm{Orb}_{\rho}(\bar{ok}))$, and we conclude the result.
\end{proof}

\begin{corollary}\label{c:sigma-inj} If the surjective homomorphism $\mathrm{Res}_{E}$ is an isomorphism, then $\Sigma$ is injective. \qed
\end{corollary}

Having considered injectivity of $\Sigma$, let us now develop a criterion for its surjectivity.  Let $\epsilon^{-1}\sigma \in \Aut(E \uparrow \pro{n})\sigma$ and define maps $\epsilon_{r}\in \Aut(E)$ by $\epsilon_{r}\defeq (\sigma^{r-1}\epsilon)(\sigma^{-1}\epsilon)^{r-1}$.

Then we have a commutative diagram
  \begin{center}
\begin{tikzpicture}[node distance=1.5cm,on grid]
\node (a) at (0,0) {$E$};
\node (b) [below=of a] {$E$};
\node (c) [right=of a] {$E$};
\node (d) [below=of c] {$E$};
\node (e) [right=of c] {$E$};
\node (f) [below=of e] {$E$};
\node (g) [right=of e] {$E$};
\node (h) [below=of g] {$E$};
\node (i) [right=of g] {$E$};
\node (j) [below=of i] {$E$};
\node (k) [right=of i] {};
\node (l) [below=of k] {};

\draw[->] (a) -- node [left] {$\id_{E}$} (b);

\draw[->] (a) -- node [above] {$\epsilon^{-1}\sigma$} (c);
\draw[->] (b) -- node [below] {$\sigma$} (d);
\draw[->] (c) -- node [left] {$\epsilon_{1}$} (d);

\draw[->] (c) -- node [above] {$\epsilon^{-1}\sigma$} (e);
\draw[->] (d) -- node [below] {$\sigma$} (f);
\draw[->] (e) -- node [left] {$\epsilon_{2}$} (f);

\draw[->] (e) -- node [above] {$\epsilon^{-1}\sigma$} (g);
\draw[->] (f) -- node [below] {$\sigma$} (h);
\draw[->] (g) -- node [left] {$\epsilon_{3}$} (h);

\draw[->] (g) -- node [above] {$\epsilon^{-1}\sigma$} (i);
\draw[->] (h) -- node [below] {$\sigma$} (j);
\draw[->] (i) -- node [left] {$\epsilon_{4}$} (j);

\draw[->] (i) -- node [above] {$\epsilon^{-1}\sigma$} (k);
\draw[->] (j) -- node [below] {$\sigma$} (l);

\end{tikzpicture}
  \end{center}
  Then if (and only if) for all $r\in \bN$ there exist maps $\hat{\epsilon}_{r}\in \Aut(\pro{n})$ such that $\hat{\epsilon}_{r}|_{E}=\epsilon_{r}$, then $\bar{\epsilon}=\{ \epsilon_{r} \}$ will be a geometric system satisfying $\Sigma(\mathrm{Orb}_{\rho}(\bar{\epsilon}))=\epsilon^{-1}\sigma$.  That is, $\Sigma$ is surjective if for all $\epsilon^{-1}\in \Aut(E \uparrow \pro{n})$ and for all $r$, we have 
\begin{equation}\label{eq:sigma-surj}
(\sigma^{r-1}\epsilon)(\sigma^{-1}\epsilon)^{r-1}\in \Aut(E \uparrow \pro{n}),
\end{equation}
or equivalently $\epsilon_{r}^{*}\cL\iso \cL$.  However this is not a practical condition to check if $\Aut(E\uparrow \pro{n})$ is not extremely small.  Note that a simple rearrangement of \eqref{eq:sigma-surj} gives an equivalent equation
\begin{equation}
\left(\sigma^{r-1}(\epsilon\sigma^{-1})^{r-1}\epsilon^{-(r-1)}\right)\epsilon^{r}\in \Aut(E \uparrow \pro{n}),
\end{equation}
which has the advantage of making it clear that if $\sigma$ and $\epsilon$ commute, the computation collapses to observing that $\epsilon^{r}\in \Aut(E\uparrow \pro{n})$, which holds by our choice of $\epsilon^{-1}$.

  Note that it is \emph{a priori} possible that an alternative construction starting from an arbitrary $\epsilon$ could yield a geometric system whose image under $\Sigma$ is $\epsilon^{-1}\sigma$; the one above is natural, of course.

In particular, we see from the above that $\Sigma$ is surjective if $\sigma^{*}\cL \iso \cL$, whence $\sigma\in \Aut(E\uparrow \pro{n})$.  However this is very strong: in dimension 3, as noted in \cite{ATV1}, if the triple $\cT(\cA)=(E,\sigma,\cL)$ coming from $\cA\in \ncpn$ has $\sigma\in \Aut(E\uparrow \pro{n})$ then the algebra associated to $\cA$ is twist-equivalent to $\cO(\pro{2})$.  But the latter situation is one we can analyse directly, as we shall do below (in all dimensions).

As remarked above, when we have a slice groupoid, such as $\ncpn/\cA$ or $\gtpn/\cT$, the only task is to understand its objects, since every Hom-set has cardinality 1.  We are now in a position to connect together
\begin{itemize}
\item the function $\mathrm{Orb}_{\rho}$ from the objects of $\gtpn/\cT(\cA)$ to the set of $\rho$-orbits $(\gtpn/\cT(\cA))\git \cP$ and
\item the function $\Sigma\colon (\gtpn/\cT(\cA))\git \cP) \to \Aut(E \uparrow \pro{n})\sigma$.
\end{itemize}
We would like an injective composite function, as then we have an upper bound for the number of twists (up to isomorphism), given by $|\Aut(E\uparrow \pro{n})\sigma|$.  We would also like to identify elements in the image, as this gives us lower bounds.

The function obtained from composing the maps listed above will very rarely be injective.  However at the various stages, we have information about the obstructions to injectivity, and we can replace our non-injective functions by injective ones, which we will do as follows.

Let $S$ be a set and $\mathcal{S}=\{ S_{i} \mid i\in I\}$ an \emph{equipartition} of $S$, that is, a partition of $S$ such that for all $i_{1},i_{2}\in I$, there exists a bijection between $S_{i_{1}}$ and $S_{i_{2}}$.  If $J$ is a set in bijection with $S_{i}$ for some (and hence for all) $i$, then it follows that there is a bijection of $S$ with $I\times J$.  Instances of equipartitions include (a) a free group action of $G$ on $S$ (trivial stabilizers imply equipartition into orbits) yielding a bijection between $S$ and $G\times (S/G)$ for $S/G$ the orbit space and (b) the partition of a group $G$ into cosets of the kernel of a group homomorphism $f\colon G\to H$ yielding a (set) bijection between $G$ and $\ker f \times \im f$.

Note that if $f\colon X\to Y$ is a function such that $\{ f^{-1}(y) \mid y\in Y\}$ is an equipartition, we have a bijection $p_{y}\colon X \to f^{-1}(y) \times \im f$ (for any choice of $y$).  Consequently\footnote{Of course, this is both elementary and fundamental (especially in the case of homomorphisms) but as it is less frequently applied to arbitrary set functions, we include the details for clarity.}, we have an injective function $\hat{p}_{y} \colon X \to f^{-1}(y) \times Y$, induced by the inclusion of $\im f$ into $Y$.  

Now we may apply this basic principle (twice) to obtain our main theorem.

\begin{theorem}\label{t:aut-ncpn} Let $\cA$ be a geometric noncommutative $\pro{n}$ and let $\cT=\cT(\cA)$.  There is an injective function
  \[ \mathrm{Obj}(\gtpn/\cT)\to \Aut(E\uparrow \pro{n})\sigma \times (\ker \mathrm{Res}_{E})^{\nat} \times \Aut(\pro{n}). \]
\end{theorem}

\begin{proof}
We claim that the functions $\mathrm{Orb}_{\rho}$ and $\Sigma$ give rise to equipartitions and hence, at the expense of altering the codomains, we may replace them by injective functions.

In the case of $\mathrm{Orb}_{\rho}$ this is immediate from Lemma~\ref{l:rho-is-free}, telling us that $\rho$ is a free action.  So $\mathrm{Orb}_{\rho}$ gives rise to an injective function
\[ \widehat{\mathrm{Orb}}_{\rho}\colon \mathrm{Obj}(\gtpn/\cT)\to ((\gtpn/\cT)\git \cP) \times \cP, \]
where we recall that $\cP=\Aut(\pro{n})$.

For $\Sigma$, Lemma~\ref{l:ker-of-Sigma} implies that $\Sigma$ determines an equipartition into the set of pre-images under $\Sigma$, since every pre-image is in bijection with $(\ker \mathrm{Res}_{E})^{\nat}$.  Then as above there is an induced bijection $(\gtpn/\cT) \git \cP \to (\ker \mathrm{Res}_{E})^{\nat} \times \im \Sigma$ and hence an induced injection
\[ \widehat{\Sigma}\colon ((\gtpn/\cT) \git \cP) \to \Aut(E \uparrow \pro{n})\sigma \times (\ker \mathrm{Res}_{E})^{\nat}. \]

Then $(\widehat{\Sigma}\times \id)\circ \widehat{\mathrm{Orb}}_{\rho}$ is the desired injection.
\end{proof}

We remark that, as previously discussed, the final term in the product that is the codomain of the injective function of the theorem---namely $\Aut(\pro{n})$---relates to choices within the isomorphism class of the algebra corresponding to $\cA$.  This follows from Proposition~\ref{p:const-morph-iso} and the observations immediately afterwards.

The term $(\ker \mathrm{Res}_{E})^{\nat}$ corresponds to non-uniqueness in extending maps from $E$ to $\pro{n}$ and the associated geometric systems obtained by taking elements of this product certainly need not be constant, so this term does not necessarily correspond to moving within an isomorphism class.  Fortunately, in examples, $\ker \mathrm{Res}_{E}$ can be explored using only the information of $E$ and its embedding into $\pro{n}$.  Heuristically, if $E$ is sufficiently large (i.e.\ has enough points in suitably general position) this kernel will be small or trivial; challenging cases here would be when $E$ is a finite collection of points.

In several important examples, to be given in detail in the next section, we see that it is reasonable to interpret the above theorem and the remarks in the preceding paragraphs to say that twists are controlled by the coset $\Aut(E\uparrow \pro{n})\sigma$, which we can describe explicitly.

\section{Examples}

We conclude with three examples, the first being the commutative algebra $\cO(\pro{n})$.  Secondly, we look at quantum projective space, $\cO_{q}(\pro{n})$.  The final example concerns 3-dimensional Sklyanin algebras, whose point schemes are smooth elliptic curves.

\begin{example}

  Consider $\cA=(\QGr \cO(\pro{n}),\pi\cO(\pro{n}),(1))$, with associated geometric triple $\cT=(\pro{n},\id_{\pro{n}},\cO(1))$.  To simplify notation, write $\id$ for $\id_{\pro{n}}$ in what follows.

  We need to identify the terms in the product $\Aut(E\uparrow \pro{n})\sigma \times (\ker \mathrm{Res}_{E})^{\nat} \times \Aut(\pro{n})$ in this case.  Firstly, $\Aut(\pro{n} \downarrow \pro{n})=\Aut(\pro{n}\uparrow \pro{n})=\Aut(\pro{n})$ and we see that $\Res_{\pro{n}}$ is the identity map. 

Now since $\id^{*}\cO(1)=\cO(1)$, $\Sigma$ is surjective.  Indeed, given $\tau^{-1}=\tau^{-1}\id\in \Aut(\pro{n} \uparrow \pro{n})\id$,
  \[ \tau_{r}=\id^{r-1}(\tau\,\id^{-1})^{r-1}\tau=\tau^{r}\in \Aut(\pro{n} \uparrow \pro{n}) \]
  so $\bar{\tau}=\{ \tau^{r} \}$ is a geometric system with $\Sigma(\mathrm{Orb}_{\rho}(\bar{\tau}))=\tau^{-1}\id$.

Hence $(\gtpn/\cT) \git \cP$ is in bijection with $\Aut(\pro{n})$ (as a set).  Since every $\rho$-orbit is itself in bijection with $\Aut(\pro{n})$, we conclude that $\gtpn/\cT$ is in bijection with $\Aut(\pro{n})\times \Aut(\pro{n})$.  We may disregard the second term if we wish to consider twists up to isomorphism, concluding that twists of $\cO(\pro{n})$ are parametrised by $\Aut(\pro{n})$, as expected.

  Note also that $\bar{\nu}=(\pro{n},\sigma',\cO(1))\to (\pro{n},\id,\cO(1))$ is in the same $\rho$-orbit as $\bar{\tau}=(\pro{n},\sigma,\cO(1))\to (\pro{n},\id,\cO(1))$ if and only if there is a constant geometric system $\bar{\mu}$ such that $\bar{\nu}=\bar{\tau}\circ \bar{\mu}$, i.e.
  \begin{center}
\begin{tikzpicture}[node distance=1.5cm,on grid]
\begin{scope}
  \node (a) at (0,0) {$\pro{n}$};
\node (b) [below=of a] {$\pro{n}$};
\node (c) [right=of a] {$\pro{n}$};
\node (d) [below=of c] {$\pro{n}$};
\node (e) [right=of c] {$\pro{n}$};
\node (f) [below=of e] {$\pro{n}$};
\node (g) [right=of e] {$\pro{n}$};
\node (h) [below=of g] {$\pro{n}$};
\node (i) [right=of g] {$\pro{n}$};
\node (j) [below=of i] {$\pro{n}$};
\node (k) [right=of i] {};
\node (l) [below=of k] {};

\draw[->] (a) -- node [left] {$\nu_{0}$} (b);

\draw[->] (a) -- node [above] {$\sigma'$} (c);
\draw[->] (b) -- node [below] {$\id$} (d);
\draw[->] (c) -- node [left] {$\nu_{1}$} (d);

\draw[->] (c) -- node [above] {$\sigma'$} (e);
\draw[->] (d) -- node [below] {$\id$} (f);
\draw[->] (e) -- node [left] {$\nu_{2}$} (f);

\draw[->] (e) -- node [above] {$\sigma'$} (g);
\draw[->] (f) -- node [below] {$\id$} (h);
\draw[->] (g) -- node [left] {$\nu_{3}$} (h);

\draw[->] (g) -- node [above] {$\sigma'$} (i);
\draw[->] (h) -- node [below] {$\id$} (j);
\draw[->] (i) -- node [left] {$\nu_{4}$} (j);

\draw[->] (i) -- node [above] {$\sigma'$} (k);
\draw[->] (j) -- node [below] {$\id$} (l);
\end{scope}

\node[rotate=-90] at (3.75,-2.5) {$=$};

\begin{scope}[yshift=-3.25cm]
  \node (a) at (0,0) {$\pro{n}$};
\node (b) [below=of a] {$\pro{n}$};
\node (c) [right=of a] {$\pro{n}$};
\node (d) [below=of c] {$\pro{n}$};
\node (e) [right=of c] {$\pro{n}$};
\node (f) [below=of e] {$\pro{n}$};
\node (g) [right=of e] {$\pro{n}$};
\node (h) [below=of g] {$\pro{n}$};
\node (i) [right=of g] {$\pro{n}$};
\node (j) [below=of i] {$\pro{n}$};
\node (k) [right=of i] {};
\node (l) [below=of k] {};
\node (m) [below=of b] {$\pro{n}$};
\node (n) [right=of m] {$\pro{n}$};
\node (o) [right=of n] {$\pro{n}$};
\node (p) [right=of o] {$\pro{n}$};
\node (q) [right=of p] {$\pro{n}$};
\node (r) [right=of q] {};

\draw[->] (a) -- node [left] {$\mu$} (b);

\draw[->] (a) -- node [above] {$\sigma'$} (c);
\draw[->] (b) -- node [below] {$\sigma$} (d);
\draw[->] (c) -- node [left] {$\mu$} (d);

\draw[->] (c) -- node [above] {$\sigma'$} (e);
\draw[->] (d) -- node [below] {$\sigma$} (f);
\draw[->] (e) -- node [left] {$\mu$} (f);

\draw[->] (e) -- node [above] {$\sigma'$} (g);
\draw[->] (f) -- node [below] {$\sigma$} (h);
\draw[->] (g) -- node [left] {$\mu$} (h);

\draw[->] (g) -- node [above] {$\sigma'$} (i);
\draw[->] (h) -- node [below] {$\sigma$} (j);
\draw[->] (i) -- node [left] {$\mu$} (j);

\draw[->] (i) -- node [above] {$\sigma'$} (k);
\draw[->] (j) -- node [below] {$\sigma$} (l);

\draw[->] (b) -- node [left] {$\tau_{0}$} (m);

\draw[->] (m) -- node [below] {$\id$} (n);
\draw[->] (d) -- node [left] {$\tau_{1}$} (n);

\draw[->] (n) -- node [below] {$\id$} (o);
\draw[->] (f) -- node [left] {$\tau_{2}$} (o);

\draw[->] (o) -- node [below] {$\id$} (p);
\draw[->] (h) -- node [left] {$\tau_{3}$} (p);

\draw[->] (p) -- node [below] {$\id$} (q);
\draw[->] (j) -- node [left] {$\tau_{4}$} (q);

\draw[->] (q) -- node [below] {$\id$} (r);
\end{scope}

\end{tikzpicture}
  \end{center}
  
  That is, there must exist $\mu\in \Aut(\pro{n})$ such that $\sigma'=\mu^{-1}\sigma\mu$.  This recovers the observation of Zhang that in this case, isomorphic twists correspond to conjugate elements of $\Aut(\pro{n})$.
\end{example}

\begin{example} Let $n\geq 2$. Recall 
that we define quantum projective $n$-space to be the $k$-algebra
  \[ \mathcal{O}_{q}(\pro{n})=k\langle x_{0},\dotsc ,x_{n} \rangle/\langle x_{i}x_{j}=qx_{j}x_{i}\ \forall\ i<j \rangle \]
  where $q\in k^{\ast}$ is assumed not a root of unity.

  Again, we wish to identify the terms in the product $\Aut(E\uparrow \pro{n})\sigma \times (\ker \mathrm{Res}_{E})^{\nat} \times \Aut(\pro{n})$ in this case.

    Define $\sqbinom{n+1}{2}=\{ I\subseteq \{ 0,\dotsc ,n\} \mid |I|=2 \}$.  By work of De Laet-Le Bruyn\footnote{See also Belmans-De Laet-Le Bruyn, \cite{BelmansDeLaetLeBruyn}.} (\cite[Proposition~1]{DeLaetLeBruyn}), the point scheme of $\cO_{q}(\pro{n})$ is the union of the lines
  \[ \ell_{I}\defeq \mathbb{V}(\{x_{j} \mid j\notin I\}),\ \text{for}\ I\in \sqbinom{n+1}{2} \] in $\pro{n}$.   Set $E=\bigcup_{I\in \sqbinom{n+1}{2}} \ell_{I}$.
  
  The associated automorphism of $E$ is defined on each line as
  \[ \sigma|_{\ell_{\{i_{1},i_{2}\}}}(0:\dotsm:0:p_{i_{1}}:0:\dotsm:0:p_{i_{2}}:0:\dotsm:0)=(0:\dotsm:0:p_{i_{1}}:0:\dotsm:0:qp_{i_{2}}:0:\dotsm:0), \]
  corresponding to the relator $x_{i_{1}}\otimes x_{i_{2}}-qx_{i_{2}}\otimes x_{i_{1}}$ via the correspondence indicated in the definition of geometricity (\ref{d:geometric}).  Indeed, the existence of such data $(E,\sigma)$ affirms that $\mathcal{O}_{q}(\pro{n})$ is geometric.

  This automorphism does not extend to all of $\pro{n}$.  Indeed, let $\underline{\lambda}=(\lambda_{I})_{I\in \sqbinom{n+1}{2}}\in (k^{\ast})^{\binom{n+1}{2}}$ and define $\mu(\underline{\lambda})\in \Aut E$ by
  \[ \mu(\underline{\lambda})|_{\ell_{\{i_{1},i_{2}\}}}((0:\dotsm:0:p_{i_{1}}:0:\dotsm :0:p_{i_{2}}:0:\dotsm:0)=(0:\dotsm:0:\lambda_{\{i_{1},i_{2}\}}p_{i_{1}}:0:\dotsm :0:p_{i_{2}}:0:\dotsm:0). \]
  Then one may check that a purported extension of $\mu(\underline{\lambda})$ to $\pro{n}$ would have to be represented by the image of the diagonal matrix
  \[ \mathrm{diag}(\prod_{k=0}^{n-1} \lambda_{k(k+1)},\dotsc ,\prod_{k=i}^{n-1} \lambda_{k(k+1)},\dotsc ,\lambda_{(n-1)n},1) \]
  in $\PGL{n+1}{k}$ subject to the conditions
  \[ \lambda_{i_{1},i_{2}}=\prod_{k=i_{1}}^{i_{2}-1} \lambda_{k(k+1)} \]
  for all $i_{1}<i_{2}$.  Now $\sigma=\mu(\underline{\lambda})$ with $\lambda_{i_{1},i_{2}}=q^{-1}$ for all $i_{1},i_{2}$, which therefore does not extend to $\pro{n}$ since $q$ is not a root of unity.

  Noting that any element $\nu$ of $\PGL{n+1}{k}$ preserving $E\subseteq \pro{n}$ must be projectively linear, send each line $\ell_{I}$ to another line $\ell_{J}$ and each intersection $e_{IJ}\defeq \ell_{I}\cap \ell_{J}$ to another such intersection, we see that $\nu$ is determined up to scalars by its values on the intersections $e_{IJ}$.  Furthermore, we must have $\ker \mathrm{Res}_{E}$ trivial, since if $\nu$ fixes $E$ pointwise, it is the identity on $\pro{n}$. By Corollary~\ref{c:sigma-inj}, $\Sigma$ is injective.

From a consideration of the remaining possibilities, it follows that
  \[ \Aut(\pro{n} \downarrow E)\iso \Aut(E \uparrow \pro{n})\iso S_{n+1}\ltimes ((k^{\ast})^{n+1}/k^{\ast}) \]
  with $S_{n+1}$ acting on $(k^{\ast})^{n+1}$ by the natural permutation action $\triangleright$: that is, for $\sigma\in S_{n+1}$, $\mu\in (k^{\ast})^{n+1}/k^{\ast}$ and $\underline{p}\in \pro{n}$ we have $(\sigma \triangleright \mu)(\underline{p})=(\sigma \circ \mu \circ \sigma^{-1})(\underline{p})$.

By explicit calculation using the above, we have checked equation~\eqref{eq:sigma-surj} and established that in that case, the image of $\Sigma$ is equal to $((k^{\ast})^{n+1}/k^{\ast})\sigma$.  This recovers the result of Mori, who has shown that $A=\cO_{(\alpha,\beta,\gamma)}(\pro{2})$ and $A'=\cO_{(\alpha',\beta',\gamma')}(\pro{2})$ are twist equivalent if and only if $\alpha'\beta'\gamma'=(\alpha\beta\gamma)^{\pm 1}$.  Recalling that $Q=(q^{-1},q,q^{-1})$ gives the single parameter quantum $\pro{2}$, $\cO_{q}(\pro{2})$, we see as a special case that $\cO_{q}(\pro{2})$ is not twist equivalent to $\cO_{q'}(\pro{2})$ for all but one other choice of $q'\in k^{\ast}$: 
\[ ((q')^{-1})\cdot q' \cdot ((q')^{-1})=(q^{-1})\cdot q \cdot (q^{-1})^{\pm 1} \qquad \iff \qquad q'=q^{\pm 1} \]
Moreover our computation shows that $\cO_{q}(\pro{2})$ has no other algebras in its twist-equivalence component other than multi-parameter quantum projective spaces; none arise from the $S_{3}$ factor so there are no non-algebraic twists in this case.  

This makes it very clear that $\ncp{2}$ has very many connected components.  It also shows that in the case of quantum algebras defined with respect to a parameter $q\in k^{\ast}$ (or indeed, multiple parameters) varying $q$ does not give rise to twist equivalences.  So while twist equivalence has some features of a deformation-theoretic problem, it is too strong and somewhat too discrete to be related to smooth variation of deformation parameters and the associated geometry of the latter.

Two natural questions remain open, awaiting further exploration.  Namely, the corresponding calculations for $\cO_{q}(\pro{n})$ and the exploration of algebras associated to other elements of $\Aut(E)$.
\end{example}

\begin{remark} Recent work of Itaba-Matsuno (\cite{ItabaMatsuno}) gives tables of defining relations for 3-dimensional quadratic AS-regular algebras.  Note that certain statements there and here are not directly comparable, due to differences in the definitions and of terminology.  For example, we reserve the term ``graded Morita equivalence'' for equivalence of the full graded module categories (with morphisms of all degrees), rather than equivalence of graded modules categories with morphisms only of degree zero (what we call twist-equivalence).
  \end{remark}

\begin{example}\label{e:Sklyanin} Let $k$ be of characteristic not equal to 2 or 3.  Consider the 3-dimensional Sklyanin algebras
  \[ \mathrm{Skl}_{3}(a,b,c)\defeq k\langle x_{0},x_{1},x_{2} \rangle / \langle ax_{i}x_{i+1}+bx_{i+1}x_{i}+cx_{i+2}^{2}\rangle \]
  where all indices are taken modulo 3.  For all but a known finite set of points $(a,b,c)$, this algebra is a geometric noncommutative $\pro{2}$ with associated point scheme the smooth elliptic curve
  \[ E\colon abc(x^{3}+y^{3}+z^{3})-(a^{3}+b^{3}+c^{3})xyz=0 \]
  and automorphism $\sigma$ given by translation by a point (with respect to the group law on $E$).  Note that again $\ker \mathrm{Res}_{E}$ is trivial: $E$ contains sufficiently many non-collinear points to ensure that elements of $\Aut(E\uparrow \pro{2})$ extend uniquely.

 A summary of relevant results and references concerning the 3-dimensional Sklyanin algebras (and their relationship with mathematical physics) may be found in work of Walton (\cite{Walton}).  For completeness, we also note that De Laet (\cite{DeLaet}) has identified an action by graded algebra automorphisms of the Heisenberg group $H_{3}$ of order 27 on Sklyanin algebras.

  By general results on elliptic curves (see for example, \cite[\S III]{Silverman}), every automorphism of $E$ is the composition of a translation and an isogeny of the curve with itself.  Thus $\Aut(E)$ is the semidirect product of the curve itself with the auto-isogeny group, the isogenies preserving a chosen base point and the translations changing the base point.

  The auto-isogeny groups of elliptic curves are well-known: depending on the $j$-invariant of the curve, the auto-isogeny group is cyclic of order $n=2$, 4 or 6
   (cf.\ \cite[Corollary III.10.2]{Silverman}).  It is common to call the auto-isogeny group $\Aut(E)$, but we shall not; rather we will denote it $\mathcal{I}(E)$.  Note that the aforementioned action is projectively linear and hence $\mathcal{I}(E)\subseteq \Aut(E\uparrow \pro{2})$.

  Furthermore, it is known when a translation extends to an automorphism of the ambient $\pro{2}$.  This is given explicitly in \cite[Lemma 5.3]{Mori}, where it is shown that a translation $\tau_{p}$ by a point $p$ extends to $\pro{2}$ if and only if $p$ is 3-torsion, i.e.\ $\tau_{p}$ has order 3.  The group $E[3]$ of 3-torsion elements of an elliptic curve $E$ is also known: it is isomorphic to $\integ/3\integ \times \integ/3\integ$, a group of order 9 (\cite[Corollary III.6.4]{Silverman}).

  Combining these results, we conclude that $\Aut(E\uparrow \pro{2})\iso (\integ/3\integ \times \integ/3\integ) \rtimes \mathcal{I}(E)$, a finite group of order $9n$.  One may interpret this as saying that Sklyanin algebras are extremely rigid, or very noncommutative, as they have very few twists.

  Note that the auto-isogenies act as graded algebra automorphisms of $\mathrm{Skl}_{3}(a,b,c)$.  Furthermore, for a point $p$, $\tau_{p}\sigma$ is again a translation by a point, so that twists of Sklyanin algebras by translations are again Sklyanin algebras. Thus, we see that $\mathscr{NC}(\mathbb{P}^{2})$ contains uncountably many connected components which consist of Sklyanin algebras, and each of these components contains finitely many isoclasses of these.
\end{example}

\bibliographystyle{halpha}
\bibliography{biblio}\label{references}

\end{document}